\documentclass[12pt]{amsart}
\usepackage{amsmath, amssymb, amsthm, latexsym}
\usepackage{tikz-cd}
\usepackage{enumitem}
\usepackage{amssymb,amscd,amsmath}
\usepackage{latexsym}
\usepackage[center]{caption}
\usepackage{tikz}

\usetikzlibrary{decorations.pathreplacing}
\usepackage{tikz}
\usepackage{mathrsfs}
\usetikzlibrary{decorations.markings}
\newcounter{braid}
\newcounter{strands}

\DeclareMathAlphabet{\bsf}{OT1}{cmss}{bx}{n}

\pgfkeyssetvalue{/tikz/braid height}{1cm}
\pgfkeyssetvalue{/tikz/braid width}{1cm}
\pgfkeyssetvalue{/tikz/braid start}{(0,0)}
\pgfkeyssetvalue{/tikz/braid colour}{black}
\pgfkeys{/tikz/strands/.code={\setcounter{strands}{#1}}}

\makeatletter
\def\cross{%
  \@ifnextchar^{\message{Got sup}\cross@sup}{\cross@sub}}

\def\cross@sup^#1_#2{\render@cross{#2}{#1}}

\def\cross@sub_#1{\@ifnextchar^{\cross@@sub{#1}}{\render@cross{#1}{1}}}

\def\cross@@sub#1^#2{\render@cross{#1}{#2}}

\def\render@cross#1#2{
  \def\strand{#1}
  \def\crossing{#2}
  \pgfmathsetmacro{\cross@y}{-\value{braid}*\braid@h}
  \pgfmathtruncatemacro{\nextstrand}{#1+1}
  \foreach \thread in {1,...,\value{strands}}
  {
    \pgfmathsetmacro{\strand@x}{\thread * \braid@w}
    \ifnum\thread=\strand
    \pgfmathsetmacro{\over@x}{\strand * \braid@w + .5*(1 - \crossing) * \braid@w}
    \pgfmathsetmacro{\under@x}{\strand * \braid@w + .5*(1 + \crossing) * \braid@w}
    \draw[braid] \pgfkeysvalueof{/tikz/braid start} +(\under@x pt,\cross@y pt) to[out=-90,in=90] +(\over@x pt,\cross@y pt -\braid@h);
    \draw[braid] \pgfkeysvalueof{/tikz/braid start} +(\over@x pt,\cross@y pt) to[out=-90,in=90] +(\under@x pt,\cross@y pt -\braid@h);
    \else
    \ifnum\thread=\nextstrand
    \else
     \draw[braid] \pgfkeysvalueof{/tikz/braid start} ++(\strand@x pt,\cross@y pt) -- ++(0,-\braid@h);
    \fi
   \fi
  }
  \stepcounter{braid}
}

\tikzset{braid/.style={double=\pgfkeysvalueof{/tikz/braid colour},double distance=1pt,line width=2pt,white}}

\newcommand{\braid}[2][]{%
  \begingroup
  \pgfkeys{/tikz/strands=2}
  \tikzset{#1}
  \pgfkeysgetvalue{/tikz/braid width}{\braid@w}
  \pgfkeysgetvalue{/tikz/braid height}{\braid@h}
  \setcounter{braid}{0}
  \let\sigma=\cross
  #2
  \endgroup
}
\makeatother

\input xypic
\newtheorem{theorem}{Theorem}
\newtheorem{proposition}[theorem]{Proposition}

\newtheorem{lemma}[theorem]{Lemma}

\newtheorem{definition}[theorem]{Definition}


\def\Z{\mathbb{Z}}

\def\C{\mathbb{C}}

\def\C{\mathbb{C}}

\def\N{\mathbb{N}}

\def\F{\mathbb{F}}

\def\md{\mathcal{D}}

\def\Zpk{\mathbb{Z}/p^{k}}
\def\Zpk1{\mathbb{Z}/p^{k-1}}

\newcommand{\rref}[1]{(\ref{#1})}

\newcommand{\beg}[2]{\begin{equation}\label{#1}#2\end{equation}}

\def\F{\mathbb{F}}

\def\sl2{\widetilde{SL_{2}(\Z)}}

\def\md
\def\rank{\operatorname{rank}}

\title[]{Howe duality over finite fields II: explicit stable computation}
\author{Sophie Kriz}
\thanks{The author was supported by a 2023 National Science Foundation
Graduate Research Fellowship, no. 2023350430}

\begin{document}

\maketitle
\vspace{-10mm}

\begin{abstract}
In this second paper of a series dedicated to type I Howe duality for finite fields,
we explicitly describe the eta and zeta correspondences constructed
in the first paper in terms of G. Lusztig's
parametrization of the irreducible characters of finite groups of Lie type
in the two so-called stable ranges. This identifies the stable eta and zeta correspondences
among the pairs of irreducible representations whose occurence with non-zero multiplicity
in the type I Howe duality correspondence was proved by S.-Y. Pan.
\end{abstract}

\vspace{-7mm}

\tableofcontents

\section{Introduction}

This is the second paper of a series dedicated to Howe duality for finite fields.
This refers to the question of how an
oscillator representation, which, over a finite field $\F_q$ (for
$q$ a power of an odd prime), forms a representation
of a symplectic group, decomposes when restricted
to a reductive dual pair, which consists of a two subgroups in the symplectic group
which are each other centralizers.
We shall continue to specifically
study the restriction of an oscillator representation to a {\em type I reductive dual pair},
consisting of a symplectic group $\text{Sp} (V)$ and an orthogonal group $\text{O} (W,B)$,
considered as subgroups of the symplectic group $\text{Sp} (V\otimes W)$
(in which $\text{Sp} (V) \times \text{O}(W,B)$ is embedded by the tensor product). 
As in the first paper of this series \cite{TotalHoweI},
we shall continue to focus on the so-called stable ranges.

The finite field context for the oscillator representation was introduced in \cite{HoweFiniteFields} by R. Howe. The first development towards determining the decomposition of its restriction to a reductive dual pair was made by J. Adams and A. Moy \cite{AdamsMoy}, who proved that 
a unipotent cuspidal representation is always tensored with another unipotent cuspidal representation
when it first occurs in some restricted oscillator representation.
This result was used by A.-M. Aubert, J. Michel, and R. Rouquier to 
form a conjecture on the behavior of the {\em unipotent} part of the restriction of an oscillator representation to a type I dual pair, based
on a decomposition they proved for the type II dual pairs
(which consist of a pair of general linear or unitary groups).

Further progress was made by S. Gurevich and R. Howe \cite{HoweG, HoweGBook}
in a different direction, considering the full oscillator representation rather than the unipotent part.
In a certain {\em stable range} of dual pairs,
they constructed a one-to-one correspondence pairing the irreducible
representations of an orthogonal group with a newly occuring irreducible representation of each
symplectic group in the stable range. They called this the {\em eta correspondence}
and constructed it using a concept of {\em rank}
which plays a role in certain dynamical questions about the representation theory
of finite groups of Lie type. 

Finally, in \cite{Pan1}, S.-Y. Pan proved the type I conjecture of Aubert, Michel, and Rouquier
by passing through a process of {\em uniform projection} to linear combinations of
Deligne-Lusztig virtual characters. 
Pan's calculation of the uniform projection of the oscillator representation's character was also used
by D. Liu and Z. Wang \cite{LiuWang} to extend the results of Adams
and Moy and describe, roughly speaking,
the behavior of the Howe correspondence with an odd orthogonal group on the
unipotent cuspidal representations of the symplectic groups.
Later, by proving a compatibility with Lusztig's parametrization of irreducible characters, Pan \cite{Pan2} classified which tensor products of irreducible representations appear with non-zero multiplicity in the restricted oscillator representation.

\vspace{2mm}

In \cite{TotalHoweI}, we constructed explicit correspondences (called the
eta and zeta correspondences, where the eta correspondence was previously defined
by S. Gurevich and R. Howe \cite{HoweG, HoweGBook}) between the sets of 
representations on the symplectic and orthogonal side in the two 
stable ranges. In this paper, we shall describe these correspondences
explicitly in terms of G. Lusztig's parametrization of 
the irreducible representations of finite groups of Lie type
(see, for example, \cite{Lusztig}).
This also tells us how the eta and zeta correspondences
fit into the list of tensor products of pairs of representations occuring with non-zero
multiplicity as identified by S.-Y. Pan \cite{Pan2}, although Pan uses a modified notation,
with which the precise dictionary will be given in the third paper of this series
(where we will also be able to discuss cases outside of the stable range).

To be more specific, consider an oscillator representation $\omega [V\otimes W]$
of a symplecitc group $\text{Sp}(V\otimes W)$
restricted to a type I reductive dual pair $(\text{Sp}(V), \text{O}(W,B))$ along the tensor
product inclusion
\beg{RedDualPairKronInc}{\text{Sp} (V) \times \text{O}(W,B) \hookrightarrow \text{Sp} (V\otimes W).}
Let us write $\widehat{G}$ for the set of irreducible complex representations of a finite group $G$.
In the first
part of this series \cite{TotalHoweI}, we proved that for pairs in the {\em symplectic stable range}
defined by the condition $\text{dim} (W) \leq \text{dim} (V)/2$,
the restriction of the oscillator representation
along \rref{RedDualPairKronInc} decomposes in terms of (twisted)
Harish-Chandra induced modules (i.e. parabolic inductions)
and a system of injections with mutually disjoint images
\beg{EtaCorrIntro}{\eta^V_{W,B}: \widehat{\text{O}(W,B)} \hookrightarrow \widehat{\text{Sp}(V)}.} 
Similarly, for $(\text{Sp}(V), \text{O}(W,B))$ in the {\em orthogonal stable range}
defined by the condition that $\text{dim} (V)$ is less than
or equal to the dimension of the maximal isotropic subspace
of $W$, the restriction of $\omega [V\otimes W]$ along \rref{RedDualPairKronInc} decomposes
in terms of (twisted) Harish-Chandra induced modules and a system of injections with mutually
disjoint images
\beg{ZetaCorrIntro}{\zeta^{W,B}_V: \widehat{\text{Sp}(V)} \hookrightarrow \widehat{\text{O}(W,B)}.}
We omit subcripts from the notation
of \rref{EtaCorrIntro}, \rref{ZetaCorrIntro} when the source is already established. 

For a
subgroup $H\subseteq G$, we will use the standard notation $\text{Ind}^G_H$
for the induction of an $H$-representation to $G$
and $\text{Res}^G_H$ for the restriction of a $G$-representation to $H$.
When $G$ or $H$ is established by the
context and there is no amiguity, we may omit the superscript or subscript from the notation.

The main result of \cite{TotalHoweI} is

\begin{theorem}\label{TotalHoweITheorem}
Let $V$ be a $2N$-dimensional symplectic space and let $W$ be an $n$-dimensional
space with symmetric bilinear form $B$. Write $h_W$ for the maximal dimension of an isotropic subspace
of $W$. Consider $(\text{Sp} (V), \text{O} (W,B))$ as a reductive dual pair in $\text{Sp} (V\otimes W)$.

\begin{enumerate}

\item If $(\text{Sp}(V), \text{O}(W,B))$ is in the symplectic stable range (meaning $\text{dim} (W) \leq
\text{dim} (V)/2$), then
there exists a system of mutually disjoint injections $\eta^V_{W,B}$ of the form \rref{EtaCorrIntro}, 
the restriction of $\omega [ V\otimes W]$ to $\text{Sp}(V) \times \text{O}(W,B)$ decomposes as
\beg{EtaCorrThmDecomp}{
\bigoplus_{k=0}^{h_W}
\bigoplus_{\rho \in \widehat{\text{O}(W[-k], B[-k])}} \eta^V (\rho) \otimes \text{Ind}_{P_k^B} (\rho \otimes \epsilon (det))
}
where $\text{Ind}_{P_k^B}$ denotes parabolic induction from the maximal parabolic $P_k^B$
in $\text{O}(W,B)$ whose Levi factor is $\text{O}(W[-k], B[-k]) \times \text{GL}_k (\F_q)$.
We consider
$\rho \otimes \epsilon (det)$ as a representation of this Levi subgroup by considering $\epsilon (det)$ as
a representation of $\text{GL}_k(\F_q)$.

\vspace{2mm}

\item If $(\text{Sp}(V), \text{O}(W,B))$ is in the orthogonal stable range (meaning $\text{dim} (V) \leq h_W$), then there exists a system
of mutually disjoint injections $\zeta^{W,B}_V$ of the form \rref{ZetaCorrIntro},
the restriction of $\omega [ V\otimes W]$ to $\text{Sp}(V) \times \text{O}(W,B)$ decomposes as
\beg{ZetaCorrThmDecomp}{
\bigoplus_{k=0}^{h_W}
\bigoplus_{\rho \in \widehat{\text{Sp}(V[-k])}}  \text{Ind}_{P_k^V} (\rho \otimes \epsilon (det)) \otimes \zeta^{W,B} (\rho)
}
where $\text{Ind}_{P_k^V}$ denotes parabolic induction from the maximal parabolic $P_k^V$
in $\text{Sp}(V)$ with Levi factor $\text{Sp}(V[-k]) \times \text{GL}_k (\F_q)$.
We consider
$\rho \otimes \epsilon (det)$ a a representation of this Levi subgroup by considering $\epsilon (det)$ as
a representation of $\text{GL}_k(\F_q)$.

\end{enumerate}
\end{theorem}

\vspace{3mm}

To state the main result of this paper, we must briefly recall
the classification of the representations symplectic and orthogonal groups, obtained from
Lusztig's parametrization of irreducible characters \cite{Lusztig}.
For a finite group of Lie type $G$, we denote by $G^*$ the dual of $G$ (see \cite{DeligneLusztig}).
Most generally, we consider the data of 
\begin{itemize}
\item The conjugacy class of a semisimple element $s$ in the dual group $G^*$.

\item A unipotent representation $u$ of the dual $(Z_{G^*} (s)^\circ )^*$ of the 
identity component of the centralizer of $s$ in $G^*$.

\end{itemize}
Here by the identity component $G^\circ$, we mean the group of elements of $G$ which are points of the identity component  (in the Zariski topology)  of the corresponding group over the algebraic closure of the ground field.
The data $[(s), u]$ corresponds to a representation of $G$
(irreducible when $(Z_{G^*} (s))^\circ = Z_{G^*} (s)$ but not necessarily otherwise)
which 
we denote by $r^G[(s), u]$ of dimension equal to the dimension
of $u$ multiplied by the prime to $q$ part of the quotient order
$|G^*/ Z_{G^*} (s)^\circ|$. Intuitively, $r^G[(s),u]$ can be thought of
as a ``faked parabolic induction" of $u$, noting that over a finite field,
there are many cases of $s$ where there is no actual
maximal parabolic with Levi factor
$Z_{G^*} (s)^\circ$. 

In each case of $G$, for every irreducible representation $\rho \in \widehat{G}$, there exists
a unique choice of data $[(s), u]$ as described above, such that $\rho \subseteq r^G[(s),u]$.
If $G$ has a connected center,
e.g. in the case of the odd special orthogonal groups
$G = \text{SO}_{2m+1} (\F_q)$, the representations $r^{G}[(s), u]$ are irreducible, and therefore
the data of a semisimple conjugacy class and a unipotent representation parametrize the irreducible
representations of $G$. In this case, we call the data $[(s), u]$ the {\em $G$-classification data}
corresponding to an irreducible representation $\rho = r^G[(s),u]$.

However, if $G$ has a disconnected center, $r^G[(s), u]$ may split further.
For example, consider the case of a symplectic group
$G= \text{Sp}_{2N} (\F_q)$ which has center $\Z/2$.
In this case, a representation $r^{\text{Sp}_{2N} (\F_q)}[(s),u]$ turns out to split if and only
if $s$ has $-1$ eigenvalues (and the coressponding factor of $u$ is non-degenerate). In this case, $r^{\text{Sp}_{2N} (\F_q)}[(s), u]$ turns out to
split into two irreducible
non-isomorphic pieces
$$r^{\text{Sp}_{2N} (\F_q)} [(s), u] = r^{\text{Sp}_{2N} (\F_q)}[(s), u , +1] \oplus 
r^{\text{Sp}_{2N} (\F_q)}[(s), u , -1].$$
Both pieces are of dimension equal to exactly half of the dimension of $r^{\text{Sp}_{2N} (\F_q)}[(s), u]$.
In this case, we call the
data $[(s), u, \pm 1]$ the {\em $\text{\em Sp}_{2N} (\F_q)$-classification data} corresponding
to an irreducible representation $\rho = r^{\text{Sp}_{2N} (\F_q)}[(s), u, \pm 1]$
and call the sign $\pm 1$ its {\em central sign}.
If $s$ has no $-1$ eigenvalues,
then as before, we call $[(s), u]$ the {\em $\text{\em Sp}_{2N} (\F_q)$-classification data} corresponding to
the irreducible representation $\rho = r^{\text{Sp}_{2N} (\F_q)}[(s), u]$.

A more complicated effect occurs for even orthogonal groups
$G= \text{O}_{2m}^\pm (\F_q)$. We discuss this case
by inducing the situation for $\text{SO}_{2m}^\pm (\F_q)$, considering the representations
$\text{Ind}_{\text{SO}_{2m}^\pm (\F_q)}^{\text{O}_{2m}^\pm (\F_q)} (r^{\text{SO}_{2m}^\pm (\F_q)} [(s),u])$, which may decompose into one, two, or four
distinct irreducible summands, depending on the eigenvalues of $s$. This effect can also be
interpreted according to certain ``sign data",
alongside the data of $(s), u$,
which we call the {\em $\text{\em O}_{2m}^\pm (\F_q)$-extended classification data} associated to an irreducible representation.
We discuss this in Section \ref{BackgroundSect}.

\vspace{3mm}

Using this description, we construct a system of disjoint injections
$$\phi^V_{W,B} : \widehat{\text{O}(W,B)} \hookrightarrow \widehat{\text{Sp}(V)}$$
$$\psi^{W,B}_V : \widehat{\text{Sp}(V)} \hookrightarrow \widehat{\text{O}(W,B)}$$
roughly defined by altering the semisimple part $(s)$ of the input representation's
classification data by
adding $-1$ eigenvalues if $W$ is odd-dimensional
and adding $1$ eigenvalues if $W$ is even-dimensional;
the unipotent part of the data is altered
by concatenating a single coordinate to the symbol corresponding to the factor
of $u$ corresponding to the altered eigenvalues $s$ to achieve
the needed new rank and defect. (Depending on the case of $(\text{Sp} (V), \text{O} (W,B))$,
there may be a choice of how to
add eigenvalues and where to concatenate the new coordinate to the correpsponding symbol,
which is determined according to the central action
of the input representation.) The purpose of
Section \ref{LusztigClaimSect} is to describe this construction.

\vspace{3mm}

The main result of this paper is that these explicit constructions in fact describe
the eta and zeta correspondences of Theorem \ref{TotalHoweITheorem}.
\begin{theorem}\label{EtaExplicitIntro}
Consider a reductive dual pair $(\text{Sp} (V), \text{O}(W,B))$ in the group $\text{Sp} (V\otimes W)$.
\begin{enumerate}
\item\label{SympStabSide}
Suppose
$(\text{Sp}(V), \text{O}(W,B))$ is in the symplectic
stable range (meaning $\text{dim} (W) \leq \text{dim} (V)/2$). Then
$$\eta^V_{W,B} = \phi^V_{W,B}.$$

\vspace{1mm}

\item\label{OrthoStabSide}
Suppose $(\text{Sp}(V), \text{O}(W,B))$ is in $\text{Sp}(V\otimes W)$ in the orthogonal
stable range (meaning $\text{dim} (V)$ is less than or equal to the maximal
dimension of a $B$-isotropic subspace of $W$). Then
$$\zeta^{W,B}_V = \psi^{W,B}_V.$$

\end{enumerate}

\end{theorem}

\noindent {\bf Remark:}
In the case of $(\text{Sp}(V), \text{O}(W,B))$ in the symplectic or orthogonal stable range, the 
decomposition we have now computed
can be used to recover S.-Y. Pan's results \cite{Pan1, Pan2} classifying
the pairs of irreducible representations of symplectic and orthogonal groups whose tensor
product appears with non-zero multiplicity in the restricted oscillator representation
$\text{Res}_{\text{Sp}(V) \times \text{O}(W,B)} (\omega [ V\otimes W])$. 
We will in fact be explicitly calculating
in our proof of Theorem \ref{EtaExplicitIntro} that each of Pan's predicted pairs
appears with multiplicity exactly one, and the resulting dimension sum adds up to
$q^{\text{dim} (V)\cdot \text{dim} (W)/2} = \text{dim} (\omega[ V\otimes W])$.

In comparison with Pan's description of the appearing pairs of irreducible representations,
our organization of the summands in terms of systems of one-to-one functions
between sets of irreducible representations of symplectic and orthogonal groups
fulfills the program of a finite field Howe duality (as originally
proposed by Howe in \cite{HoweFiniteFields, HoweKobayashi, HoweG, HoweGBook}).
We shall explicitly compare our decomposition with Pan's result in the upcoming paper \cite{TotalHoweIII},
where we treat the restricted oscillator representation's decomposition for general $(\text{Sp} (V), 
\text{O} (W,B))$.

\vspace{3mm}

The main tool used to prove Theorems \ref{EtaExplicitIntro} is
actually dimension. We state here a key result, which, combined with some combinatorics,
will prove that the dimensions of the $\text{Sp}(V)$-representations
$\eta^V_{W,B}(\rho)$ and $\phi^V_{W,B} (\rho)$ (resp. the $\text{O}(W,B)$-representations
$\zeta^{W,B}_V(\rho)$ and $\psi^{W,B}_V (\rho)$) always match
for $\rho \in \widehat{\text{O}(W,B)}$ (resp. $\rho \in \widehat{\text{Sp}(V)}$) for
$N>>n$ (resp. $n>>N$). 
We can derive then that they must always match, since for a fixed $\rho$, the dimensions
of $\eta^V_{W,B} (\rho)$ and $\phi^V_{W,B} (\rho)$ both form polynomials of $q^{N}$ (resp. the
dimensions of $\zeta^{W,B}_V (\rho)$ and $\psi^{W,B}_V (\rho)$ both form
polynomials of $q^n$). 
In combination with the fact that the semisimple part
of the classification data and the sign data of $\eta^V_{W,B} (\rho)$
or $\zeta^{W,B}_V( \rho)$
are already determined
by considering the restriction of the oscillator
representation to the general linear group, this will suffice to prove the representations themselves match.

We define, for $\rho$ an irreducible representation of $\text{Sp}_{2N} (\F_q)$, its {\em $N$-rank} to be
$$rk_N (\rho) = \lceil \frac{deg_q (\text{dim} (\rho))}{N} \rceil .$$
Similarly, for $\rho$ an irreducible representation of $\text{O} (W,B)$ with $\text{dim} (W) = n$, define
its {\em $n$-rank} to be
$$rk_n (\rho)= \lceil \frac{deg_q (\text{dim} (\rho))}{n} \rceil .$$

\begin{proposition}\label{NnRank}
Assume the notation of Theorem \ref{TotalHoweITheorem}.

\begin{enumerate}
\item \label{N>>nPartOfRankProp} Consider $N>>n$. Then the
disjoint union of the images of the eta correspondences
$$\eta^V_{W,B}: \widehat{\text{O}(W,B)} \hookrightarrow \widehat{\text{Sp}(V)}$$
for the symplectic space $V$ of dimension $2N$ and
the two choices of orthogonal space $(W,B)$ of dimension $n$
is precisely the set of irreducible representations of $\text{Sp}(V)$ of $N$-rank $n$.

\item \label{n>>NPartOfRankProp} Consider $n>>N$. Then the
image of the zeta correspondences
$$\zeta^V_{W,B}: \widehat{\text{Sp}(V)} \hookrightarrow \widehat{\text{O}(W,B)} $$
for the symplectic space $V$ of dimension $2N$ and
an orthogonal space $(W,B)$ of dimension $n$
is precisely the set of irreducible representations of $\text{O}(W,B)$ of $n$-rank $N$.

\end{enumerate}
\end{proposition}

The present paper is organized as follows: In Section \ref{BackgroundSect},
we describe the classification of irreducible representations
obtained from Lusztig's parametrization of characters.
In Section \ref{LusztigClaimSect}, we describe our proposed constructions of the eta and
zeta correspondences in more detail. In Section \ref{CombinatoricsSect}, we prove
combinatorial identities proving our claimed constructions can be plugged into 
the decompositions in Theorem \ref{TotalHoweITheorem} and add up to the correct dimension.
In Section \ref{InductiveProof}, we use an inductive argument to prove Proposition \ref{NnRank}
and conclude Theorem \ref{EtaExplicitIntro}.
In Section \ref{SL2Sect}, we write down the zeta correspondence in the example
where $\text{dim} (V) = 2$.
\vspace{3mm}

\noindent {\bf Acknowledgement:} The author is thankful to the GAP package
CHEVIE which was used
to verify the results of this paper in cases of small rank.

\section{The Classification of Irreducible Representations}\label{BackgroundSect}

The purpose of this section is to give more details about
the classification of irreducible representations obtained
from Lusztig's parametrization of characters, specifically
in the case
of the symplectic and orthogonal groups.

We consider a reductive group $G$ over $\mathbb{F}_q$, specifically with a focus
on the case of symplectic, orthogonal, and special orthogonal groups.
The dual group $G^*$ can be constructed as a reductive group (also over $\F_q$)
whose a roots are obtained as the coroots dual to the original roots of $G$.
For example, 
$$\begin{array}{c}
\text{Sp}_{2r}^* = \text{SO}_{2r+1}, \; \text{SO}_{2r+1}^* = \text{Sp}_{2r}, \; (\text{SO}_{2r}^\pm )^* = \text{SO}_{2r}^\pm .
\end{array}$$
The role of the dual group is so that the data,
up to conjugacy, of a pair of a
maximal torus $T$ in $G$ and a choice of an irreducible character $\theta$ on $T(\F_q)$
defines a well-defined conjugacy class
of a semisimple element $(s)$ in $G^*$.
For more details, see Section 5 of \cite{DeligneLusztig}.

For such a pair of a maximal torus $T\subseteq G$ and a character $\theta$,
one can construct a virtual character $R_T(\theta)$ of $G(\F_q)$
called the {\em Deligne-Lusztig induction} of $\theta$ 
(see, for example, \cite{DeligneLusztig, LusztigDLDisconn, DeligneLusztigDisconnected}).
For the purpose
of this paper, we will treat this construction as a black box. Every irreducible representation
of the group $G(\F_q)$ appears with a non-zero coefficient in some Deligne-Lusztig induction.
The irreducible representations of $G(\F_q)$ can be partitioned into
disjoint subsets of $\widehat{G(\F_q)}$ indexed
by the geometric conjugacy classes of a torus $T$ and a character $\theta$ 
in whose Deligne induction they appear (Corollary 6.3 of \cite{DeligneLusztig}).
In the language of the dual group, we may associate each of these subsets of $\widehat{G(\F_q)}$
called the Lusztig series corresponding to a semisimple conjugacy class $(s) \in G^*(\F_q)$.

The irreducible representations in the Lusztig series for $(1) \in G^* (\F_q)$ are called the {\em unipotent
irreducible representations} of $G$. 
We shall write $\widehat{G}_u$ for the set of irreducible unipotent representations
of $G$. We note that there is a bijection between the irreducible unipotent
representations of a group $G$ and its dual, which we denote by
$$\begin{array}{c}
\widehat{G}_u \rightarrow \widehat{(G^*)}_u\\
u \mapsto \widetilde{u}\hspace{5mm}
\end{array}
$$

To consider the case of $G=\text{Sp}_{2r}$ or $\text{SO}_{2r}^\pm$
which have disconnected center, these series can be further partioned
according to elements of $Z_{G^*(\F_q)} (s)/ Z_{G^*(\F_q)} (s)^\circ$, which in these cases
may be $\mu_2 = \{\pm 1\}$ as we will discuss later (see \cite{DeligneLusztigDisconnected}).

Next, each Lusztig series corresponding to $(s) \in G^* (\F_q)$ 
(and possibly a sign when $Z_{G^*(\F_q)} (s)/ Z_{G^*(\F_q)} (s)^\circ = \mu_2$)
can be identified in bijective correspondence with the unipotent
irreducible representations of the (dual of the) centralizer $Z_{G^* (\F_q)} (s)^\circ$
(see for example 
\cite{Lusztig} Theorem 4.23, \cite{DeligneLusztigDisconnected}
Proposition 3.4).

We now summarize the classification of the irreducible representations of
$G= \text{Sp}_{2r}$, $\text{SO}_{2r+1}$, $\text{SO}_{2r}^\pm$ obtained from this theory,
as indexed by {\em $G$-classification data}, consisting of 
\begin{itemize}
\item (``semisimple data"): a conjugacy class $(s)$
of a semisimple element $s$ of the dual group $G^*(\F_q)$

\item (``unipotent data"): a unipotent representation $u$
of the dual of the identity component $Z_{G^* (\F_q) } (s)^\circ$ of $s$'s centralizer $Z_{G^* (\F_q) } (s)$.

\item (for $G= \text{Sp}_{2r}$ or $\text{SO}_{2r}^\pm$,
possible ``central data"): a specification of
a sign $\pm 1$ when $Z_{G^*(\F_q)} (s)/ Z_{G^*(\F_q)} (s)^\circ = \mu_2$.
\end{itemize}
For every choice of this data, we denote the associated irreducible
representation by 
$r^G[(s),u ]$ when no central data occurs and $r^G[(s), u , \pm 1]$ when central data does occur.
Its dimension is the dimension of $u$ multiplied by the prime to $q$ part
of the quotient of the order of $G$ over the order of $s$'s centralizer (see \cite{DigneMichel} Remark 13.24,
\cite{LusztigDLDisconn}).
This can be expressed in our cases as
\beg{FinalDimIrredGenForm}{\begin{array}{c}
\text{dim} (r^G [(s), u] ) = \frac{|G|_{q'}}{|Z_{G^*(\F_q)} (s)^\circ|_{q'}} \text{dim}(u), \\[1ex]
 \text{dim} (r^G [(s), u, \pm 1] ) = \frac{|G(\F_q)|_{q'}}{2 |Z_{G^*(\F_q)} (s)^\circ|_{q'}} \text{dim}(u)
\end{array}}
where $|?|_{q'}$ denotes the prime to $q$ part of the group order.
In the case when central data occurs, we also may consider the sum
$$r^G [(s), u]  = r^G [(s), u, +1]  \oplus r^G [(s), u, -1] .$$

Finally, for our purposes, we will also need to consider cases of
the full orthogonal groups $G = \text{O}_{2m+1}$ and $G= \text{O}_{2m}^\pm $.
Each of these $G$'s is an extension of its identity component $G^\circ$ (i.e. the corresponding
special orthogonal group) by $\mu_2$.
We approach the full orthogonal groups
$G= \text{O}_{2r+1}$ and $\text{O}_{2r}^\pm$ by 
indexing the irreducible representations by the corresponding special orthogonal
group $G^\circ$'s classifiction data and
\begin{itemize}
\item (``extension data"): an index $\gamma$ specifying of which irreducible summand
of the induction of a representation of the corresponding special orthogonal group
we are referring to.
\end{itemize}
More specifically,
when $G= \text{O}_{2r+1} (\F_q)$, the index $\gamma$ is an element of $\{\pm 1\}$.
When $G= \text{O}_{2r}^\pm (\F_q)$, $\gamma$ is an element of $\{\pm 1\}^{a(s)+b(s)}$
where
$$a(s) = \begin{cases}
1 & \parbox{4in}{if $s$ has the eigenvalue $1$ and the factor of $u$ corresponding to the
eigenvalue $1$ is a non-degenerate Lusztig symbol}\\[2ex]
0 & \text{else}
\end{cases}$$
$$b(s) = \begin{cases}
1 & \parbox{4in}{if $s$ has the eigenvalue $-1$ and the factor of $u$ corresponding to the
eigenvalue $-1$ is a non-degenerate Lusztig symbol}\\[3ex]
0 & \text{else.}
\end{cases}$$
Note that in the case of $G= \text{O}_{2r}^\pm (\F_q)$, $(s)$ is an element
of $\text{SO}_{2r}^\pm (\F_q)$.

The purpose of this section is to give more detail about each part of the classification data
in each case of $G$ we consider in this paper.
In Subsection \ref{ToriSubSect}, we discuss the maximal tori
in symplectic and orthogonal groups and the form of the semisimple elements,
up to conjugation. In Subsection \ref{IrredOddSO}, we discuss the irreducible representations
of $\text{O}_{2m+1} (\F_q)$. 
In Subsection \ref{EvenSOIrredSubSect}, we discuss the irreducible representations of $\text{O}_{2m}^\pm
(\F_q)$.
In Subsection \ref{IrredSP}, we discuss the irreducible representations
of $\text{Sp}_{2N} (\F_q)$.

\subsection{Tori and the case of rank 1}\label{ToriSubSect}

Suppose $G (\F_q)$ is of the form
$\text{Sp}_{2r} (\F_q)$, $\text{SO}_{2r+1} (\F_q)$, or $\text{SO}_{2r}^\pm (\F_q)$.
The maximal tori in $G(\F_q)$ are all conjugate to a product of $\text{SO}_2^\pm (\F_{q^n})$
factors
\beg{ToriInAll}{\text{SO}_{2}^\pm (\F_{q^{n_1}}) \times \dots \times \text{SO}_{2}^\pm (\F_{q^{n_k}})}
of maximal rank, so that
$r = n_1+ \dots + n_k,$
and where, in the case of $G$ an even special orthogonal group $\text{SO}_{2r}^\pm (\F_q)$,
the sign in the superscript is equal to the product of the signs appearing in \rref{ToriInAll}.
Recall that
\beg{SO2+}{\text{SO}_2^+ (\F_q) = \{ \begin{pmatrix}
x & 0\\
0 & x^{-1}
\end{pmatrix} \mid x \in \F_q^\times \} \cong \mu_{q-1}}
\beg{SO2-}{\text{SO}_2^- (\F_q) = \{ \begin{pmatrix}
y & z\\
\varepsilon z & y
\end{pmatrix} \mid y,z \in \F_q \; y^2 - \varepsilon z^2 = 1\}\cong \F_{q^2}^\times /\F_q^\times \cong \mu_{q+1},}
where in \rref{SO2-}, $\varepsilon \in \F_q^\times $
is an element which is not a square, and the isomorphism 
follows by considering $\F_{q^2} = \F_q [ \sqrt{\varepsilon}]$, whose norm $1$ elements are isomorphic to
$\mu_{q+1}$. 
Note that the eigenvalues of an element 
$\begin{pmatrix} y & z\\
\varepsilon z & y
\end{pmatrix}$ of \rref{SO2-} are precisely the conjugate norm $1$ elements $y\pm \sqrt{\varepsilon}z$.

Every semisimple element $(s) \in G (\F_q)$ is in the $G(\F_q)$-conjugacy class of an element of \rref{ToriInAll}, which is classified by the orbit of the
list of its eigenvalues, which can lie in $\mu_{q-1} \cong \F_q^\times$
or the norm $1$ elements $\mu_{q+1}$ of $\F_{q^2}^\times$, under the action
of the Weyl group of $G$.

By changing the coordinates of the underlying orthogonal or symplectic form 
corresponding to $G$, we can in fact consider the tori \rref{ToriInAll} as
embedded in $G$ directly by taking a direct sum of matrices.
We note that in the case of $G$ an odd 
special orthogonal group $\text{SO}_{2r+1} $, to embed
a torus of the form \rref{ToriInAll} into $G$, we need to insert a ``forced" diagonal entry $1$
to obtain a matrix of size $2r+1$. In other words, every semisimple element
of $\text{SO}_{2r+1} (\F_q)$ has a single extra $1$ eigenvalue in addition to the eigenvalues
detected by its conjugacy class in \rref{ToriInAll}, so that the total number of $1$ eigenvalues
can be odd.
The placement
of this entry is according to whether the product \rref{ToriInAll} is a subgroup of $\text{SO}_{2r}^+ (\F_q)$
or $\text{SO}_{2r}^- (\F_q)$. 

Therefore, the data of a list of elements
\beg{SemisimplEigenvalData}{(\lambda_1, \dots , \lambda_t) \in \prod_{i=1}^t \mu_{q^{r_i\pm 1}}}
for $\lambda_i \in \mu_{q^{r_i} \pm 1} $ determines a semisimple element
of $\text{SO}_{2r}^\pm (\F_q)$ obtained by taking a direct sum
of matrices in $\text{SO}_{2r_i}^\pm (\F_q)$
corresponding to each $\lambda_i$,
so that in $\text{SO}_{2r_i}^\pm  (\F_{q^{r_i}})$, each is conjugate to
\beg{BlocksHigherFieldExtSemisimpIntro}{\bigoplus_{j = 0}^{r_i -1} \begin{pmatrix}
\lambda_i^{q^j} & 0\\
0  & \lambda_i^{-q^j}
\end{pmatrix}.}
To make this form easier to discuss, let us introduce a notation for these blocks:
Write $A_{\lambda_i}$ for the element of $\text{SO}_{2r}^\pm (\F_q)$ conjugate
to \rref{BlocksHigherFieldExtSemisimpIntro}.

This is enough information about semisimple elements
for our purposes in the cases of even special orthogonal groups
and symplectic groups. We note that in the case of symplectic groups, we may embed each
factor $\text{SO}_2^\pm (\F_{q^{r_i}})$ of the torus into $\text{SL}_2 (\F_{q^{r_i}})$, in which
we may consider the Weyl group action,
so the conjugacy class corresponding to \rref{SemisimplEigenvalData} only depends on
the equivalence class
\beg{SemisimplEigenvalDataEquiv}{(\lambda_1, \dots , \lambda_t) \in \prod_{i=1}^t \mu_{q^{r_i} \pm 1}/ (\lambda \sim \lambda^{-1}).}

In the case of odd special
orthogonal groups $\text{SO}_{2r+1} (\F_q)$, there is one more subtlety.
As described above, the semisimple elements of $\text{SO}_{2r+1} (\F_q)$ correspond to semisimple
elements of $\text{SO}_{2r}^\pm (\F_q)$, embedded into $\text{SO}_{2r+1} (\F_q)$
by adding a single diagonal $1$ entry forced by the symmetric bilinear form.
For $2r$ by $2r$ matrices which can be considered
in groups \rref{ToriInAll} for more than one choice of signs (meaning some
factors are equal to the identity matrix $I$ or $-I$), then we must consider whether these two choices
of where to insert the final forced diagonal entry $1$ give different conjugacy classes, or not.
The only cases of eigenvalues which can correspond to elements of either 
choice of torus splitness are $1$ and $-1$ eigenvalues.
If only $1$ eigenvalues are present, it is indistinguishable where we add the extra $1$
diagonal entry. So we find that these two choices
give different conjugacy classes if and only if the
$2r$ by $2r$ element considered in \rref{ToriInAll}
has any $-1$ eigenvalues, in which case the resulting two choices
of elements will turn out to have different centralizers and cannot be conjugate.
We also note that odd special orthogonal groups, like symplectic groups,
do have a large enough Weyl group to consider eigenvalue data
only up equivalence class as in \rref{SemisimplEigenvalDataEquiv}.

\vspace{2mm}

To summarize, we may consider any semisimple element of a group $G$
as conjugate to a sum of blocks
\beg{SAsSumOfBlocks}{s \sim A_{\lambda_1} \oplus \dots \oplus A_{\lambda_t},}
with an additional $1$ inserted in the case of $G=\text{SO}_{2r+1} (\F_q)$, 
considering the extra data of a choice
of where precisely if one of the $\lambda_i$ is equal to $-1$.
Further, if we must refer to a maximal torus containing $s$,
let us also always choose to minimize the field extensions $\F_{q^{r_i}}$ needed
in \rref{ToriInAll} to contain the eigenvalues of $s$, so that there are no $r_i'< r_i$ such that
$\lambda_i \in \F_{q^{r_i'}} \subset \F_{q^{r_i}}$.

\begin{definition}
For $G$ a symplectic or (special) orthogonal group, we say a semisimple element $s$ is in
a {\em generic conjugacy class} if it has no $\pm 1$ eigenvalues
(not counting the forced $1$ eigenvalue for $G$ odd (special) orthogonal). Otherwise, say $s$'s
conjugacy class is {\em singular
of type $(p, \ell )$} if it corresponds to a sum of blocks of the from
$$(A_1)^{\oplus p} \oplus (A_{-1})^{\oplus \ell} \oplus \bigoplus_{i=1}^s A_{\lambda_i}$$
as in \rref{SAsSumOfBlocks}, for $\lambda_i \neq \pm 1$ (again disregarding the forced $1$ eigenvalue for $G$ odd (special) orthogonal).
\end{definition}

To see how our description of semisimple elements works in practice,
let us discuss the examples of rank 1 groups $G$. The even cases of
$\text{SO}_2^\pm (\F_q) = \mu_{q\mp 1}$ are not large enough to show any interesting effects,
and are abelian.

\vspace{1mm}

\noindent{\bf Example: $\mathbf{G} =\mathbf{SL}_2 (\F_q)$ and $\mathbf{SO}_3 (\F_q)$}
{\em For either case of $G = \text{SL}_2 (\F_q)$ or $\text{SO}_3 (\F_q)$, the only
maximal tori are isomorphic to the special orthogonal groups
$\text{SO}_2^\pm (\F_q) \sim \mu_{q\mp 1}$.

On the one hand, in the case of $G=\text{SL}_2 (\F_q)$, there are
\begin{enumerate}
\item two central elements 
$$A_1 = \begin{pmatrix} 1 & 0\\ 0 &1 \end{pmatrix}, \;
A_{-1}=\begin{pmatrix} -1 & 0\\ 0 &-1 \end{pmatrix},$$
of types $(1,0)$ and $(0,1)$ respectively, which can be considered in both $\text{SO}_2^\pm (\F_q)$.

\item $(q-3)/2$ generic semisimple conjugacy classes with representatives
$$A_\lambda = \begin{pmatrix} \lambda & 0\\ 0 &\lambda^{-1} \end{pmatrix} \sim  
\begin{pmatrix} \lambda^{-1} & 0\\ 0 &\lambda \end{pmatrix}$$
for $\lambda \in \F_q^\times \smallsetminus \{\pm 1\}$

\item $(q-1)/2$ generic semisimple conjugacy classes with representatives $A_\mu \in \text{SO}_2^- (\F_q)$ which are conjugate to
$$\begin{pmatrix} \mu & 0\\ 0 &\mu^{-1} \end{pmatrix} \sim  
\begin{pmatrix} \mu^{-1} & 0\\ 0 &\mu \end{pmatrix}$$
in $\text{SO}_2^- (\F_{q^2})$,
for $\mu \in \mu_{q+1} \smallsetminus \{\pm 1\}$
\end{enumerate}
In total, there are $q$ semisimple conjugacy classes in $\text{SL}_2 (\F_q)$.

On the other hand, in $G=\text{SO}_3 (\F_q)$, we must consider more carefully
how the blocks $A_\lambda$, $A_\mu$ can be embedded as $3\times 3$ matrices in $G$.
For this, let us suppose the symmetric bilinear form $B$ defining $G = \text{SO}(\F_q^3, B)$
is of the form 
$$B = \begin{pmatrix} 1 & 0 & 0\\ 0 & -1 & 0 \\ 0 & 0 & a \end{pmatrix}$$
with discriminant $-1$ (the case of discriminant $1$ is entirely
similar, with a reversed choice of where the forced $1$'s are placed).
For a choice of $\lambda \in \F_q^\times$, corresponding to a block $A_\lambda \in \text{SO}_2^+ (\F_q)$,
we can embed it as the element
$$
A_\lambda^+ := \begin{pmatrix}
\lambda & 0 & 0\\
0 & \lambda^{-1} & 0\\
0 & 0 & 1
\end{pmatrix}.
$$
(due to the ambiguity in the case when $\lambda = -1$, it is meaningful in this case to keep track
of which sign is in the superscript of the choice of $\text{SO}_2^\pm (\F_q)$ we start with).
To see why $A_\lambda^+$ and $A_{\lambda^{-1}}^+$ are conjugate, we have
$$
\begin{pmatrix}
0 & 1& 0\\
-1 & 0 &0\\
0 & 0 & 1
\end{pmatrix}\cdot \begin{pmatrix}
\lambda & 0 & 0\\
0 & \lambda^{-1} & 0\\
0 & 0 & 1
\end{pmatrix} \cdot \begin{pmatrix}
0 & -1 & 0\\
1 & 0 & 0\\
0 & 0 & 1
\end{pmatrix}=
\begin{pmatrix}
\lambda^{-1} & 0 & 0\\
0 & \lambda & 0\\
0 & 0 & 1
\end{pmatrix}.$$
Similarly, for choices of $\mu = y+\sqrt{\varepsilon} z$ of norm $1$ in $\F_{q^2}$, we can
embed $A_\mu \in \text{SO}_2^- (\F_q)$ as a $3\times 3$ matrix in $\text{SO}_3 (\F_q)$ as
$$A_\mu^- := \begin{pmatrix}
y & 0 & z\\
0 & 1 & 0\\
\varepsilon z & 0 & y
\end{pmatrix}$$
A similar argument gives that $A_\mu^-$ and $A_{\mu^{-1}}^-$ are conjugate in $\text{SO}_3 (\F_q)$.
For $\lambda \in \F_q^\times \smallsetminus \{\pm 1\}$ and $\mu \in \mu_{q+1} \smallsetminus \{\pm1\}$, we obtain $(q-3)/2$ and $(q-1)/2$ generic semisimple conjugacy classes, respectively.


Now consider the elements $A_{-1}^+$ and $A_{-1}^-$.
There are two singular conjugacy classes of type $(0,1)$: the
conjugacy classes of
$$\sigma_1^+ := \begin{pmatrix}
-1 & 0 & 0\\
0 & -1 & 0\\
0 & 0 & 1
\end{pmatrix}, \hspace{3mm}
\sigma_1^-:= \begin{pmatrix}
-1 & 0 & 0\\
0 & 1 & 0\\
0 & 0 & -1
\end{pmatrix},
$$
which have centralizers $\text{O}_2^+ (\F_q)$, $\text{O}_2^- (\F_q)$ in $\text{SO}_3 (\F_q)$, respectively
(since here
the centralizer is the orthogonal group on the two coordinates corresponding
to the two $-1$ entries in $\sigma_1^\pm$).
We therefore find that there
are a total of $q+1$ semisimple conjugacy classes in $\text{SO}_3 (\F_q)$.
}

\vspace{1mm}
In $\text{SO}_{2r+1} (\F_q)$,
elements $\sigma_n^\pm$ which play a role generalizing $\sigma_1^\pm$ exist and
play a very important role in considering oscillator representations.
We define them now:

\begin{definition}\label{SigmanDefn}
Let us consider an odd special orthogonal group $\text{SO}_{2r+1} (\F_q)$.
Consider the maximal tori obtained by emebedding
$$\begin{array}{c}
(\text{SO}_2^+ (\F_q))^r \subseteq \text{SO}_{2r}^+ (\F_q)\\
(\text{SO}_2^+ (\F_q))^{r-1} \times \text{SO}_2^- (\F_q) \subseteq \text{SO}_{2r}^- (\F_q)\\
\end{array}$$
into $\text{SO}_{2r+1} (\F_q)$. Consider the element consisting of a sum of $r$ copies of $A_{-1}$
lying in either choice of torus in $\text{SO}_{2r}^\pm (\F_q)$.
We write
$$ \sigma^+_r := ( A_{-1}^{\oplus r})^+, \hspace{5mm} \sigma^-_r := ( A_{-1}^{\oplus r})^-$$
for the $(2r+1)\times (2r+1)$ matrices in $\text{SO}_{2r+1} (\F_q)$
obtained by adding the $1$ forced by considering the sum of $A_{-1}$'s
as an element of the split and non-split special orthogonal group, respectively
(as in the above example, the only role of the superscript $\pm$ is to record the sign
of $\text{SO}_{2r}^\pm (\F_q)$ we consider $A_{-1}^{\oplus r}$ in).
\end{definition}

We find that the centralizers of these semisimple elements 
are the even orthogonal groups
\beg{OscRepsCentras}{
Z_{\text{SO}_{2r+1} (\F_q)} (\sigma_r^\pm) \cong \text{O}_{2r}^\pm (\F_q), \hspace{3mm}
Z_{\text{SO}_{2r+1} (\F_q)} (\sigma_r^\pm )^\circ \cong \text{SO}_{2r}^\pm (\F_q)}
In particular $\sigma_r^+$ and $\sigma_r^-$ cannot be conjugate in $\text{SO}_{2r+1} (\F_q)$.

\subsection{The representation theory of the symplectic group}\label{IrredSP}
We now describe the classification data
for the irreducible representations of $\text{Sp}_{2N }(\F_q)$.
The choices of semisimple data consist of semisimple conjugacy
classes in the dual group  $\text{Sp}_{2N}^* (\F_q) = \text{SO}_{2N+1} (\F_q)$.
Fix a semisimple conjugacy class
$(s) \in \text{SO}_{2N+1} (\F_q)$ and say it is conjugate to a sum of blocks
\beg{SO2N+1ABlocksCase}{s\sim A_{1}^{\oplus p} \oplus (A_{-1}^{\oplus \ell})^\alpha \oplus \bigoplus_{i=1}^r A_{\lambda_i}^{\oplus j_i} \oplus \bigoplus_{i=1}^t A_{\mu_i}^{\oplus k_i}}
for distinct choices of eigenvalues 
\beg{DistinctNontrivEigenvs}{\begin{array}{c}
\lambda_i \in \F_{q^{n_i'}}^\times \smallsetminus \{\pm 1\} \text{ for } i =1,\dots ,r\\[1ex]
 \mu_i \in \mu_{q^{n_i''} +1} \smallsetminus \{\pm 1\}\text{ for } i =1,\dots ,t
\end{array}
}
and where $\alpha$ denotes the sign for which we consider $A_{-1}^{\oplus  \ell} \in \text{SO}_{2\ell}^\alpha (\F_q)$. There is no condition on this sign. The size of the matrix requires $N$
\beg{Rankmatrizsizesemisimpdata}{p + \ell + \sum_{i=1}^r j_i n_i' + \sum_{i=1}^t k_i n_i'' .}
Note that $s$ is then of type $(p,\ell)$ and generic when $p = \ell = 0$.
Then the identity component of the
centralizer of $s$ in $\text{SO}_{2N+1} (\F_q)$ is isomorphic to the product
\beg{SO2lPmInCentForSO2N+1}{\prod_{i=1}^r U_{j_i}^+ (\F_{q^{n_i'}}) \times\prod_{i=1}^t U_{k_i}^- (\F_{q^{n_i''}}) \times \text{SO}_{2\ell}^\pm (\F_q)
\times \text{SO}_{2p+1}(\F_q).}
(We use the notation that $U_j^+ (\F_q) = \text{GL}_j (\F_q)$.)
The full centralizer of $s$ is obtained by replacing the factor $\text{SO}_{2\ell}^\pm (\F_q)$
by $\text{O}_{2\ell}^\pm (\F_q)$. In particular, note that
$Z_{\text{SO}_{2N+1} (\F_q)} (s) /Z_{\text{SO}_{2N+1} (\F_q)} (s)^\circ$ is $\mu_2$
precisely when $\ell>0$ and is trivial when $\ell =0$.


Given such semisimple data $(s) \in \text{SO}_{2N+1} (\F_q)$, let us consider
an irreducible unipotent representation $u$ of the dual group of
\rref{SO2lPmInCentForSO2N+1}
(taking the dual replaces the odd special orthogonal group factor by the corresponding
symplectic group of the same rank and leaves all other factors the same). 
Then $u$ can be expressed as a tensor product
\beg{SymplecticRepsUnipFactors}{\bigotimes_{i=1}^r u_{U^+_{j_i}} \otimes\bigotimes_{i=1}^t u_{U^-_{k_i}} \otimes u_{\text{SO}_{2\ell}^\pm}  \otimes u_{\text{Sp}_{2p}}}
where $u_{U_{j_i}^+}$, $u_{U_{k_i}^-}$,  $u_{\text{SO}_{2\ell }^\pm}$, and $u_{\text{Sp}_{2 p }}$
are unipotent irreducible representations of $U_{j_i}^+ (\F_{q^{n_i'}})$, $U_{k_i}^- (\F_{q^{n_i''}})$,
$\text{SO}_{2\ell}^\pm (\F_q)$, and $\text{Sp}_{2p} (\F_q)$, respectively.

The irreducible unipotent representations of these factors
can be further described using the theory of {\em symbols}
\cite{LusztigSymbols}.
We do not discuss the case of the unipotent representations of a finite group of Lie type
$A$ or ${}^2 A$ here.
For now, we consider the symbols of type $B$, $C$, $D$, and ${}^2D$ of rank $r$.

\begin{definition}\label{SymbolsDefn}
Consider the data of an equivalence classes of two rows of strictly increasing sequences
\beg{TwoRowSymbol}{
{\lambda_1 < \lambda_2 < \dots < \lambda_a
\choose \mu_1< \mu_2 < \dots < \mu_b}
}
under switching rows, for $\lambda_i, \mu_i \in \N_0 $ non-negative integers 
such that $(\lambda_1, \mu_1) \neq (0,0)$.
\begin{enumerate}
\item \label{OddSymbolsItem}
The data \rref{TwoRowSymbol} is called a {\em symbol of rank $r$ of type $C$ or $B$} if 
\beg{RankBCtypeSymbols}{ \sum_{i=1}^a\lambda_i + \sum_{i= 1}^b \mu_i = r + \frac{(a+b-1)^2}{4},}
and the {\em defect} $a-b$ is odd.

\item \label{EvenSymbolIsItem}
The data \rref{TwoRowSymbol} is called a {\em symbol of rank $r$ of type $D$} (resp. {\em of type ${}^2 D$}) if
\beg{RankDtypeSymbols}{
 \sum_{i=1}^a \lambda_i + \sum_{i=1}^b \mu_i =r +\frac{(a+b)(a+b-2)}{4}
}
and the {\em defect} $a-b$ is $0$ mod $4$ (resp. $2$ mod $4$). We call a
symbol of type $D$ {\em degenerate} if the two rows match, i.e. $a= b$ and $\lambda_i = \mu_i$.
\end{enumerate}
\end{definition}

The symbols of type $B$ or $C$ of rank $r$ classify the irreducible unipotent representations
of the group $\text{SO}_{2r+1} (\F_q)$ or $\text{Sp}_{2r} (\F_q)$
(since these groups have the same Weyl group, their irreducible unipotent representations have the same classification). The symbols of type $D$, resp. ${}^2 D$, of rank $r$ classify the irreducible unipotent
representations of the group $\text{SO}_{2r}^+ (\F_q)$, resp. $\text{SO}_{2r}^- (\F_q)$.

The dimension of the unipotent representation of $G$ corresponding
to a non-degenerate symbol \rref{TwoRowSymbol} is the factor
\beg{SymbolDimTwoRowsBigFrac}{
\frac{\displaystyle \prod_{1\leq i<j \leq a} (q^{\lambda_j} - q^{\lambda_i}) \cdot
\prod_{1\leq i<j \leq b} (q^{\mu_j} -q^{ \mu_i}) \cdot \prod_{\text{$\begin{array}{c}
1\leq i \leq a\\[-1.75ex]
1\leq j \leq b
\end{array}$}} (q^{\lambda_i} + q^{\mu_j})  }{\displaystyle 
\prod_{i=1}^a  \prod_{j=1}^{\lambda_i} (q^{2j} -1)
\cdot \prod_{i=1}^b  \prod_{j=1}^{\mu_i} (q^{2j} -1)\cdot q^{c[a,b]} }}
multiplied by $|G|_{q'}/
2^{\lfloor (a+b-1)/2\rfloor}$,
where, for the final power of $q$ in the denominator of \rref{SymbolDimTwoRowsBigFrac}, we write
$$ c[a,b] = \sum_{i=1}^{\lfloor (a+b)/2 \rfloor} {a+b-2i \choose 2} .$$

In particular, the dimension of the unipotent represent associated to a non-degenerate symbol 
${\lambda_1< \dots <\lambda_a \choose \mu_1 < \dots < \mu_b}$ of type
$D$ of ${}^2 D$ and rank $p$
is again the factor \rref{SymbolDimTwoRowsBigFrac},
multiplied by
\beg{DTypeFactorSymbol}{\frac{|SO_{2p}^\pm (\F_q)|_{q'}}{2^{(a+b-2)/2}}.}
Each degenerate symbol of $\text{SO}_{2r}^+ (\F_q)$ corresponds to a pair of distinct non-isomorphic
irreducible unipotent representations of $\text{SO}_{2r}^+ (\F_q)$, of dimension equal to
\rref{SymbolDimTwoRowsBigFrac} multiplied
by $|\text{SO}_{2r}^+ (\F_q)|_{q'}/ 2^{a+b}$ (i.e. by an additional half compared
to the factor for non-degenerate symbols).

\vspace{2mm}

For choices of semisimple and unipotent data $(s)$ of type $(p, 0)$ and $u$,
we can form irreducible $\text{Sp}_{2N} (\F_q)$-representations
$r^{\text{Sp}_{2N} (\F_q)}[(s), u]$ of dimension equal to the dimension of $u$ multiplied by the factor
\beg{indexSp2N}{\frac{|\text{Sp}_{2N} (\F_q)|_{q'}}{|\prod_{i=1}^r U_{j_i}^+ (\F_{q^{n_i'}}) \times\prod_{i=1}^t U_{k_i}^- (\F_{q^{n_i''}}) \times
\text{SO}_{2\ell}^\pm (\F_q) \times \text{Sp}_{2p} (\F_q)|_{q'}}.}

Say $s$ is of type $(p,\ell)$ for $\ell \neq 0$
and $u_{\text{SO}_{2\ell}^\pm }$ corresponds to a non-degenerate symbol.
In this case, for each $\alpha \in \{\pm 1\}$ in the quotient of $s$'s centralizer
over its identity component, there is an irreducible representation
$r^{\text{Sp}_{2N} (\F_q)}[(s), u, \alpha]$
of dimension equal to half the dimension of $u$ times
\rref{indexSp2N}.
%
In the case where $s$ has $-1$ as an eigenvalue (i.e. $\ell \neq 0$) and
$u_{\text{SO}_{2\ell}^\pm }$ corresponds to a degenerate symbol, only a single irreducible
representation $r^{\text{Sp}_{2N} (\F_q)} [(s) ,u]$ of $\text{Sp}_{2N} (\F_q)$ is produced
(i.e the same irreducible $\text{Sp}_{2N} (\F_q)$-representation is produced by using the other irreducible
unipotent of $\text{SO}_{2\ell}^\pm$ corresponding to the same symbol, ``merging" the two degenerate pieces).
In this case, we obtain a single irreducible representation
$r^{\text{Sp}_{2N} (\F_q)} [(s) ,u]$
of dimension equal to $2$ times $\text{dim} (u)$ times \rref{indexSp2N}.


\vspace{2mm}

\noindent {\bf Example:}
The oscillator representations 
\beg{OscRepsIntoHalves}{\omega_a = \omega_a^+ \oplus \omega_a^-,
\hspace{5mm} \omega_b = \omega_b^+ \oplus \omega_b^-}
of $\text{Sp}_{2N} (\F_q)$ can be recovered using this classification.
%
%
%
%
%
Take $\sigma_N^\pm \in \text{SO}_{2N+1} (\F_q)$ as the semisimple data.
Recalling \rref{OscRepsCentras}, the identity component of the centralizer
of $\sigma_N^\pm$ is $\text{SO}_{2N}^\pm (\F_q)$, which is self-dual.
The factor \rref{indexSp2N} is then
\beg{QNpm1IndexFact}{\frac{|\text{Sp}_{2N} (\F_q)|_{q'}}{|\text{SO}^\pm_{2N} (\F_q)|_{q'}} = 
\frac{\prod_{i=1}^N (q^{2i} -1)}{(q^N \mp 1)\prod_{i=1}^{N-1} (q^{2i}-1)}
=q^N\pm 1}
In both cases, take $u$ to be the trivial representation $1$.

The representations $r^{\text{Sp}_{2N} (\F_q)}[(\sigma_N^\pm), 1]$ of dimension $q^N\pm 1$ decompose according to the action of the center $\Z/2 = Z (\text{Sp}_{2N} (\F_q))$
$$r^{\text{Sp}_{2N} (\F_q)}[(\sigma_N^+), 1] =  \omega_a^+ \oplus \omega_b^+,
\hspace{5mm} r^{\text{Sp}_{2N} (\F_q)}[(\sigma_N^-), 1] =  \omega_a^- \oplus \omega_b^-$$
recovering the irreducible components of the oscillator representations \rref{OscRepsIntoHalves}. 
In our notation, we have
$\omega_{a}^\pm = r^{\text{Sp}_{2N} (\F_q)}[(\sigma_N^+), 1, \epsilon(a)]$,
where $\epsilon$ denotes the quadratic character 
\beg{QuadCharacter}{\epsilon: \F_q^\times \rightarrow \{\pm 1\}.}

\subsection{The representation theory of the odd orthogonal groups}\label{IrredOddSO}
Next, we describe the classification data for the irreducible representations of the odd orthogonal groups $\text{O}_{2m+1} (\F_q)$. In this case, we can split the center off
$$\text{O}_{2m+1} (\F_q) =  \Z/2 \times \text{SO}_{2m+1} (\F_q),$$ 
and therefore each irreducible representation can be considered as the tensor product
of a sign with its irreducible restriction to 
the special orthogonal group $\text{SO}_{2m+1} (\F_q)$.
Now since $\text{SO}_{2m+1} (\F_q)$ has no center,
the irreducible representations are precisely the representations $r^{\text{SO}_{2m+1} (\F_q)}[(s),u]$,
corresponding to choices of $\text{SO}_{2m+1} (\F_q)$-classification data consisting of
a conjugacy class $(s)$ of a semisimple element 
$s \in \text{Sp}_{2m} (\F_q) = \text{SO}_{2m+1}^* (\F_q)$,
and $u$ an irreducible unipotent representation of the dual of
the identity component of the centralizer of $s$ in $\text{Sp}_{2m} (\F_q)$. We call the sign defining the action of the 
center $\mu_2 = Z (\text{O}_{2m+1} (\F_q))$ the $\text{O}_{2m+1} (\F_q)${\em -extension data} corresponding
to a certain irreducible rerpesentation and write
\beg{ROExtSignData}{r^{\text{O}_{2m+1} (\F_q)} [ (s), u]^{\pm 1}: = (\pm 1) \otimes r^{\text{SO}_{2m+1} (\F_q)} [ (s), u].}
It suffices to describe the classification data $[(s), u]$ describing the irreducible
$\text{SO}_{2m+1} (\F_q)$-representations.

The choices of semisimple data for $\text{SO}_{2m+1} (\F_q)$ consist of
semisimple conjugacy classes in the dual group
$\text{SO}_{2m+1}^* (\F_q) = \text{Sp}_{2m} (\F_q)$.
Fix such an $(s) \in \text{Sp}_{2m} (\F_q)$ and, similarly as in the previous subsection,
say it is conjugate to a sum of blocks
\beg{SemiSimpElt}{s\sim A_{1}^{\oplus p} \oplus A_{-1}^{\oplus \ell} \oplus \bigoplus_{i=1}^r A_{\lambda_i}^{\oplus j_i} \oplus \bigoplus_{i=1}^t A_{\mu_i}^{\oplus k_i}}
for distinct eigenvalues $\lambda_i$, $\mu_i$ as in \rref{DistinctNontrivEigenvs}.
Again the matrix size gives an expression of $m$ as the sum
of products of field extension degrees and multiplicities \rref{Rankmatrizsizesemisimpdata}.
The centralizer of $s$ in $\text{Sp}_{2m} (\F_q)$ is
\beg{CentralizerinSp2mFq}{\prod_{i=1}^r U^+_{j_i} (\F_{q^{n_i'}}) \times \prod_{i=1}^t U^-_{k_i}
 (\F_{q^{n_i''}}) \times
\text{Sp}_{2\ell} (\F_q) \times \text{Sp}_{2p} (\F_q).}
Note that this is always connected for any choice of $(s)$.

Given a choice of such semisimple data $(s)\in \text{Sp}_{2m} (\F_q)$,
let us consider
an irreducible unipotent representation $u$ of the dual of $s$'s connected
centralizer \rref{CentralizerinSp2mFq}
(taking the dual replaces each symplectic group by the corresponding odd special orthogonal group of the
same rank and leaves all other factors the same).
Then $u$ may be expressed as a tensor product of unipotent representations
of each factor of \rref{CentralizerinSp2mFq}
\beg{UnipotentRepSp2mFqCentra}{\bigotimes_{i=1}^r u_{U^+_{j_i}} \otimes\bigotimes_{i=1}^t u_{U^-_{k_i}} \otimes u^{-1}_{\text{SO}_{2\ell +1}} \otimes u^{+1}_{\text{SO}_{2p+1}},}
where again, $u_{U^+_{j_i}}$, $u_{U^-_{k_i}}$,
$u^{-1}_{\text{SO}_{2\ell +1}}$, and $u^{+1}_{\text{SO}_{2p +1}}$ are unipotent irreducible representations of
$U^+_{j_i} (\F_{q^{n_i'}})$, $U^-_{k_i}(\F_{q^{n_i''}})$,
$\text{SO}_{2\ell+1} (\F_q)$, and $\text{SO}_{2p+1} (\F_q)$, respectively
(the superscript for the $C$-type factors indicates the sign of the eigenvalue $\pm 1$
of the blocks in $s$ which the factor corresponds to).
Both $u_{\text{SO}_{2\ell+1}}^{-1}$ and $u_{\text{SO}_{2p+1}}^{+1}$ correspond to symbols of type
$B$ of rank $\ell$ and $p$ respectively, as defined in Definition \ref{SymbolsDefn},
part (\ref{OddSymbolsItem}).

The irreducible $\text{SO}_{2m+1}(\F_q)$ representation $r^{\text{SO}_{2m+1} (\F_q)}[ (s), u]$ corresponding to 
such a choice of $(s)$, $u$ has dimension equal to the dimension of $u$ multiplied by the
prime to $q$ part of the quotient of orders
\beg{indexSo2m+1}{\frac{|\text{SO}_{2m+1} (\F_q)|_{q'}}{| \prod_{i=1}^r U^+_{j_i} (\F_{q^{n_i'}}) \times \prod_{i=1}^t U^-_{k_i} (\F_{q^{n_i''}}) \times
\text{SO}_{2\ell+1} (\F_q) \times \text{SO}_{2p+1} (\F_q)|_{q'}}.}
According to \rref{ROExtSignData}, for both choices of extension sign data $\pm 1$, 
$$\text{dim} (r^{\text{O}_{2m+1} (\F_q)} [(s), u]^{\pm 1}) =
\text{dim} (r^{\text{SO}_{2m+1} (\F_q)} [(s), u]).$$ 
Call the collection of data $[(s), u]$ together with an extension sign $\pm 1$
{\em $\text{\em O}_{2m+1} (\F_q)$-extended classification data}.

\subsection{The representation theory of the even orthogonal groups}\label{EvenSOIrredSubSect}
Finally, we consider the special orthogonal group on a $2m$-dimensional $\F_q$-vector
space $W$ with respect to a symmetric bilinear form $B$. Write
$\text{O}(W,B) = \text{O}^\alpha_{2m} (\F_q)$,
with sign $\alpha = +$ if $B$ is totally split (i.e. there is an $m$-dimensional
isotropic subspace of $W$), and with sign $\alpha = -$ otherwise. We call this sign $\alpha$
the {\em total sign} of the symmetric bilinear form $B$.

To classify an irreducible representation $\rho \in \widehat{\text{O}_{2m}^\alpha (\F_q)}$,
consider
the semisimple and unipotent $\text{SO}_{2m}^\alpha (\F_q)$-classification data $(s)$ and $u$
such that $\rho$ is a summand of the corresponding
representation's induction to $\text{O}^\alpha_{2m} (\F_q)$
\beg{IndSO2mAlph}{\rho \subseteq \text{Ind}_{\text{SO}_{2m}^\alpha (\F_q)}^{\text{O}^\alpha_{2m} (\F_q)} (
r^{\text{SO}_{2m}^\alpha (\F_q)}[(s),u]).}
The induction on the right hand side of \rref{IndSO2mAlph} may decompose into $1$, $2$, or
$4$ irreducible representations, which we enumerate according
to {\em $\text{\em O}_{2m}^\alpha (\F_q)$-extension data}. We note that unlike in the odd orthogonal choice
there is not an obvious natural choice of how to do this.

Therefore the $\text{O}_{2m}^\alpha (\F_q)$-classification
data can be considered as to consist of semisimple part the conjugacy class of $s$
as an element of $\text{O}_{2m}^\alpha (\F_q)$, unipotent part $u$, and this extension
sign data.
(Note that the unipotent representations of a group are the same after removing the center,
and for simplicity to compare with the
odd case, we may consider unipotent parts of classification data as irreducible
unipotent representations of $\text{SO}_{2m}^\alpha (\F_q)$.)

\vspace{1mm}

Let us begin by describing the special orthogonal group's classification data.
The group $\text{SO}_{2m}^\alpha (\F_q)$ is its own dual, so the semisimple component of its
classification data only consists of a semisimple conjugacy class $(s) \in \text{SO}_{2m}^\alpha (\F_q)$.
Say $(s)$ is conjugate to a sum of blocks
\beg{sSO2mAlphaBlocks}{s\sim A_{1}^{\oplus p} \oplus A_{-1}^{\oplus \ell} \oplus \bigoplus_{i=1}^r A_{\lambda_i}^{\oplus j_i} \oplus \bigoplus_{i=1}^t A_{\mu_i}^{\oplus k_i}}
for eigenvalues as in \rref{DistinctNontrivEigenvs},
ranks such that $m$ equals \rref{Rankmatrizsizesemisimpdata}.
%
%
%
%
%
%
Then the identity component of the centralizer of $s$ is
\beg{EvenCentralizerChoices}{
\prod_{i=1}^r U_{j_i}^+ (\F_{q^{n_i'}}) \times \prod_{i=1}^t U_{k_i}^- (\F_{q^{n_i''}})
\times \text{SO}_{2 \ell}^\pm(\F_q) \times \text{SO}_{2p}^\pm (\F_q)}
(where, the total product of signs appearing in \rref{EvenCentralizerChoices} is the total sign of $B$).

Given such semisimple data $(s) \in \text{SO}_{2m}^\alpha (\F_q)$,
the unipotent part of the $\text{SO}_{2m}^\alpha (\F_q)$-classification data
consists of a unipotent irreducible representation $u$ of the identity component of $s$'s centralizer
\rref{EvenCentralizerChoices}, which can be
factored as
\beg{EvenUnipotentFactorization}{ \bigotimes_{i=1}^r u_{U_{j_i}^+} \otimes \bigotimes_{i=1}^t u_{U_{k_i}^-}
\otimes u_{\text{SO}_{2\ell}^\pm}^{-1} \otimes u_{\text{SO}_{2p}^\pm}^{+1},}
using the same notation as in the case of the odd special orthogonal
groups.

Again, the $\text{SO}_{2m}^\pm (\F_q)$-representation $r^{\text{SO}_{2m}^\pm (\F_q)} [(s), u]$
corresponding to such a choice of $(s), u$ is of dimension $\text{dim} (u)$ multiplied by
\beg{}{\begin{array}{c}
|\text{SO}_{2m}^\pm (\F_q)|_{q'}/ |Z_{\text{SO}_{2m}^\pm (\F_q)} (s)^\circ |_{q'} = \\[1ex]
\displaystyle \frac{|\text{SO}_{2m}^\pm (\F_q)|_{q'}
}{|\prod_{i=1}^r U_{j_i}^+ (\F_{q^{n_i'}}) \times \prod_{i=1}^t U_{k_i}^- (\F_{q^{n_i''}})
\times \text{SO}_{2\ell}^\pm (\F_q) \times \text{SO}_{2p}^\pm (\F_q)|_{q'}.}
\end{array}}
The unipotent representations $u_{\text{SO}_{2\ell}^\pm}^{-1}$,
$u_{\text{SO}_{2p}^\pm}^{+1}$ correspond to the symbols of type $D$ or ${}^2 D$ of ranks
$\ell$ and $p$.


The dimension of the unipotent represent associated to a symbol 
${\lambda_1< \dots <\lambda_a \choose \mu_1 < \dots < \mu_b}$ of type
$D$ of ${}^2 D$ and rank $p$
is again the factor \rref{SymbolDimTwoRowsBigFrac},
multiplied by
\beg{DTypeFactorSymbol}{\frac{|SO_{2p}^\pm (\F_q)|_{q'}}{2^{(a+b-2)/2}}}
(in the case of $SO_{2p}^+ (\F_q)$, if the two rows of the symbol are exactly the same,
it is called {\em degenerate}, and splits into two additional equi-dimensional non-isomorphic halves).

%

It remains to describe the decomposition of the induction
\beg{InduceSemisimplUnipotentSO2r2O2r}{\text{Ind}_{\text{SO}_{2m}^\pm (\F_q)}^{\text{O}^\pm_{2m} (\F_q)} (
r^{\text{SO}_{2m}^\pm (\F_q)}[(s),u])}
First, we notice that as long as $(s)$ has some eigenvalues not equal to $\pm 1$, there
are choices of semisimple data $(s')$ for $\text{SO}_{2m}^\pm (\F_q)$ such that $s$ and $s'$ are not
conjugate in the special orthogonal group, but they are conjugate in $\text{O}_{2m}^\pm (\F_q)$
(precisely since replacing $A_\lambda$ by $A_{\lambda^{-1}}$).
For such cases, we find
$$\text{Ind}_{\text{SO}_{2m}^\pm (\F_q)}^{\text{O}_{2m}^\pm (\F_q)} ( r^{\text{SO}_{2m}^\pm (\F_q)} [(s), u])
\cong \text{Ind}_{\text{SO}_{2m}^\pm (\F_q)}^{\text{O}_{2m}^\pm (\F_q)} ( r^{\text{SO}_{2m}^\pm (\F_q)} [(s'), u]).
$$

We also see another effect. Inducing an irreducible unipotent representation splits
into two non-isomorphic equidimensional irreducible unipotent representations
$$\begin{array}{c}
 \text{Ind}_{\text{SO}_{2m}^\pm (\F_q)}^{\text{O}_{2m}^\pm (\F_q)} ({\lambda_1 < \dots < \lambda_a
\choose \mu_1< \dots < \mu_b}) = \\[1ex]
 {\lambda_1 < \dots < \lambda_a
\choose \mu_1< \dots < \mu_b}^{+} \oplus {\lambda_1 < \dots < \lambda_a
\choose \mu_1< \dots < \mu_b}^{-}.
\end{array}$$
We note that, in the split case,
for both irreducible $\text{SO}_{2m}^+ (\F_q)$-summands of the unipotent representation
corresponding to the degenerate symbol ${\lambda_1 < \dots < \lambda_a \choose \lambda_1 <\dots < \lambda_a}$,
their inductions to an $\text{O}_{2m}^+ (\F_q)$ are irreducible and isomorphic.
Ultimately, we find that \rref{InduceSemisimplUnipotentSO2r2O2r}
splits into $2^{a (s) +b (s)}$ irreducible equidimensional $\text{O}_{2m}^\pm (\F_q)$-representations,
where we put
$$\begin{array}{c}
a (s) = \begin{cases}
1, \text{ if } 1 \text{ is an eigenvalue of } s \text{ and } u_{SO_{2p}^\pm}^1 \text{ is non-degenerate}\\
0, \text{ else}
\end{cases}\vspace{1mm}\\[1ex]
b (s) = \begin{cases}
1, \text{ if } -1 \text{ is an eigenvalue of } s \text{ and } u_{SO_{2\ell}^\pm}^{-1} \text{ is non-degenerate}\\
0, \text{ else.}
\end{cases}
\end{array}$$
To summarize, we may enumerate the irreducible representations of $\text{O}_{2m}^\pm (\F_q)$
according to the {\em $\text{O}_{2m}^\alpha (\F_q)$-extended classification data}
consisting of 
\begin{enumerate}
\item The {\em semisimple data}
of the conjugacy class of a semisimple element $s \in \text{SO}_{2m}^\pm (\F_q)$,
under conjugacy by elements in the full orthogonal group $\text{O}_{2m}^\pm (\F_q)$.

\item The {\em unipotent data}
of a unipotent representation of the dual of the identity component of $s$'s centralizer
$(Z_{\text{SO}_{2m}^\pm (\F_q) }(s)^\circ)^*$, consisting of a tensor product
of symbols, allowing the 
possible degenerate symbols (which we do not decompose in order to avoid 
over-counting).

\item The {\em extension sign data} $\gamma$, consisting of $a (s) + b (s)$
independent choices of sign. If $s$ has both
$+1$ and $-1$ eigenvalues, we write $\gamma = (\pm 1, \pm 1)$, listing the sign associated to
the presence of $1$ eigenvalues first. When only one of $a(s)$ and $b (s)$ is
non-zero, we write $\gamma $ as the single choice of sign itself. When both $a (s)$
and $b (s)$ are $0$.
\end{enumerate}
We denote the corresponding irreducible $\text{O}_{2m}^\alpha (\F_q)$-representation by
$r^{\text{O}_{2m}^\alpha (\F_q)} [ (s), u]^\gamma$.

%
%

\section{The claimed construction}\label{LusztigClaimSect}

The purpose of this section is to describe in more detail the claimed
construction that we are proposing gives the eta and zeta correspondences.
This will identify the ``top" partner of an input representation appearing
in the Howe correspondence of S.-Y. Pan \cite{Pan1, Pan2} (we discuss this in \cite{TotalHoweIII} in more detail).
We define the proposed constructions
$$\phi^V_{W,B}: \widehat{\text{O}(W,B)} \hookrightarrow \widehat{\text{Sp}(V)}$$
in the symplectic stable range, and
$$\psi^{W,B}_V: \widehat{\text{Sp}(V)} \hookrightarrow \widehat{\text{O}(W,B)}$$
in the orthogonal stable range.

We recall now that 
the unipotent irreducible representations of a
group $G$ are the same as the unipotent irreducible representations of its dual $G^*$.
For an irreducible unipotent representation
$u$ of $G$, we denote by $\widetilde{u}$ the corresponding irreducible unipotent representation
of $G^*$.

In Subsection \ref{OddSympStSubSect}, we treat the case of 
$(\text{Sp}(V), \text{O}(W,B))$ in the symplectic 
symplectic stable range where the dimension of $W$ is odd.
In Subsection \ref{EvenSympStSubSect}, we treat the case of the
symplectic stable range where the dimension of $W$ is even.
In Subsection \ref{OddOrthoStSubSect}, we treat the case of 
a reductive dual pair $(\text{Sp}(V), \text{O}(W,B))$ in the orthogonal
stable range where the dimension of $W$ is odd.
In Subsection \ref{EvenOrthoStSubSect}, we treat the case of the
orthogonal stable range where the dimension of $W$ is even.

\subsection{The odd symplectic stable case}\label{OddSympStSubSect}
Consider a choice of type I reductive dual pair
$(\text{Sp}(V), \text{O}(W,B))$ in the symplectic stable range
for odd-dimensional $W$. Write $\text{dim} (V) = 2N$, $\text{dim} (W) = 2m+1$. The range condition
requires that $N \geq 2m+1$. 

In this case, the center splits off of the orthogonal group, and
we may consider $\text{O}(W,B) = \Z/2 \times \text{SO}_{2m+1} (\F_q)$.
Our goal is to define a construction whose input is an irreducible representation of
$\pi = r^{\text{O}_{2m+1} (\F_q)} [(s), u]^{\pm 1}$ of $\text{O}(W,B)$ 
for $O_{2m+1} (\F_q)$-extended classification data
consisting of a semisimple conjugacy class 
$(s)$ in the dual group $\text{SO}_{2m+1}^* (\F_q) = \text{Sp}_{2m} (\F_q)$,
an irreducible unipotent
representation $u$ of $(Z_{\text{Sp}_{2m} (\F_q)} (s)^\circ )^*$, and a sign $\pm 1$;
the output should be a unique irreducible representation $\phi^V_{W,B} (\pi)$
of $\text{Sp}(V) = \text{Sp}_{2N} (\F_q)$.
In other words, we must produce from $(s)$, $u$ and the sign $\pm 1$, a choice of new 
$\text{Sp}_{2N} (\F_q)$-classification data:
\beg{NewLusztigClassDataOddSympSt}{[(\phi^\pm (s)), \phi^\pm (u), \text{disc}(B) \cdot \varepsilon (s)]}
(consisting of semisimple $(\phi^\pm (s)) \in \text{Sp}_{2N}^* (\F_q) = \text{SO}_{2N+1} (\F_q)$
and an irreducible unipotent representation $\phi^\pm (u)$ of its 
the dual of the identity component of its centralizer).
Broadly, we construct $\phi^\pm (s)$ by adding $-1$ eigenvalues to $s$, and we alter the symbol
in the affected factor of the unipotent part by adding a single coordinate to one of the rows
(to obtain the needed rank and defect).

\vspace{3mm}


We begin with describing the semisimple part $(\phi^\pm (s))$. Considering
$s$ as an element of a maximal torus of the form \rref{ToriInAll}, and it is determined by the data
of the orbit of its eigenvalues.
Recalling Definition
\ref{SigmanDefn}, we then define $\phi^\pm (s)$ to be the semisimple element
$$\begin{array}{c}
\displaystyle \phi^\pm (s) := s\oplus \sigma_{N-m}^\pm \in \prod_{i=1}^k \text{SO}_{2}^\pm (\F_{q^{n_i}})
\times \text{SO}_{2(N-m)+1} (\F_q) \\
\vspace{1mm} \subseteq \text{SO}_{2N+1} (\F_q) = \text{Sp}_{2N }^* (\F_q).
\end{array}$$
On the level of eigenvalues, $\phi^\pm (s)$ are obtained precisely by adding $2(N-m)$ eigenvalues
$-1$ and a single $1$ eigenvalue to $s$'s original eigenvalues, in a position where projecting away from
the coordinate of the $1$ eigenvalue gives a subspace of $W$ where $B$ is completely split
if the sign is $+$ and non-split if the sign if $-$.
Suppose $s$ is of type $(p, \ell)$, i.e. $s$
has eigenvalue $-1$ of multiplicity $2\ell$, and $1$ of multiplicity $2p$, and the identity component of
its centralizer is of the form \rref{CentralizerinSp2mFq}.
For simplicity, let us separate the factors corresponding to the eigenvalues of $s$ not equal to $-1$
$$H= \prod_{i=1}^r U_{j_i}^+ (\F_{q^{n_i'}}) \times \prod_{i=1}^t U_{k_i}^- (\F_{q^{n_i''}})  \times \text{Sp}_{2p} (\F_q),$$
so that $Z_{\text{Sp}_{2m} (\F_q)} (s)^\circ = H \times \text{Sp}_{2\ell} (\F_q)$.
The identity component of the centralizer of this element $\phi^\pm (s)$ in $\text{SO}_{2N+1} (\F_q)$ is precisely
\beg{IdComponentZSO2N+1PhiPmHSOEv}{\begin{array}{c}
\displaystyle Z_{\text{SO}_{2N+1} (\F_q)} (\phi^\pm (s))^\circ = 
 H^* \times \text{SO}_{2(N-m+\ell)}^\pm (\F_q).
\end{array}}

Now let us describe the unipotent part $\phi^\pm (u)$ of \rref{NewLusztigClassDataOddSympSt}.
Let us factor $u$ as in \rref{UnipotentRepSp2mFqCentra}, and let us consider ths symbol
$ u_{\text{SO}_{2\ell+1}}^{-1} = {\lambda_1< \dots < \lambda_a \choose \mu_1< \dots < \mu_b}$
and write $u_H$ for the unipotent $H$-representation consisting of a product of the other factors $\bigotimes_{i=1}^r u_{U^+_{j_i}} \otimes\bigotimes_{i=1}^t u_{U^-_{k_i}} \otimes u^{+1}_{\text{SO}_{2p+1}}$,
so that we can write
$u = u_H \otimes {\lambda_1< \dots < \lambda_a \choose \mu_1< \dots < \mu_b}$.
The defect $a-b$ of the symbol  is odd, so we may switch rows to assume
without loss of generality that $a-b$ is $1$ mod $4$.
We define a certain natural number associated to $\pi = r^{\text{O}_{2m+1} (\F_q)} [(s), u]^{\pm 1}$
$$N'_{\pi} = N-m + \frac{a+b-1}{2}$$
(note that by the symplectic stable range condition, we automatically have $N'_{\pi}
\geq N-m \geq m+1$).
Then we may concatenate $N'_{\pi}$ onto the end of either row of the symbol
${\lambda_1< \dots < \lambda_a \choose \mu_1< \dots < \mu_b}$, obtaining new symbols
$$\textstyle\phi^+ ({\lambda_1< \dots < \lambda_a \choose \mu_1< \dots < \mu_b}) = 
{\lambda_1< \dots < \lambda_a \choose \mu_1< \dots < \mu_b < N'_{\pi}},$$
describing a unipotent representation of $\text{SO}_{2(N-m+\ell)}^+ (\F_q)$ and
$$\textstyle  \phi^- ({\lambda_1< \dots < \lambda_a \choose \mu_1< \dots < \mu_b}) = 
{\lambda_1< \dots < \lambda_a < N'_\pi \choose \mu_1< \dots < \mu_b },$$
describing a unipotent representation of $\text{SO}_{2(N-m+\ell)}^- (\F_q)$.
We then put
$$\textstyle \phi^\pm (u) := \widetilde{u_H} \otimes \phi^\pm ({\lambda_1< \dots < \lambda_a \choose \mu_1< \dots < \mu_b}),$$
giving a unipotent representation of the group \rref{IdComponentZSO2N+1PhiPmHSOEv}.

Finally, we need a central sign to complete the classification
data \rref{NewLusztigClassDataOddSympSt} since by definition $\phi^\pm (s)$ has $-1$
eigenvalues.
Consider again $s$ as an element of the torus \rref{ToriInAll}.
Further, consider each factor
$\text{SO}_2^\pm (\F_{q^{r_i}}) \cong \mu_{q^{r_i} \mp 1}$.
Then define $\epsilon (s)$ to be the product of applying the quadratic character on
each $\Z/ (q^{r_i } \mp 1)$ to each coordinate, giving a total sign.
Multiplying with the discriminant $\text{disc}(B)$ gives the central sign of \rref{NewLusztigClassDataOddSympSt}.

\begin{definition}Given the above notation we define
$\phi^V_{W,B} (\pi)$ to be the irreducible $\text{Sp}(V)$-representation with the new
classification data we constructed:
$$\phi^{V}_{W,B} (r^{\text{\em O}(W,B)}[(s),u]^{\pm 1}) := r^{\text{\em Sp} (V)} [(\phi^\pm (s)), \phi^\pm (u),
\text{\em disc}(B) \cdot \epsilon (s)].$$
\end{definition}

\subsection{The even symplectic stable case}\label{EvenSympStSubSect}
Now suppose the reductive dual pair
$(\text{Sp}(V), \text{O}(W,B))$ is in the symplectic stable case and
$W$ is even dimensional. Write $\text{dim} (V) = 2N$ and $\text{dim} (W) = 2m$.
Write $\alpha$ for the sign
so that $\text{O} (W,B) = \text{O}^\alpha_{2m} (\F_q)$. In both cases, the orthogonal stable range
condition requires that $N \geq 2m$.

Fix an input irreducible $\text{O}_{2m}^\alpha (\F_q)$-representation $\pi = r^{\text{O}_{2m}^\alpha (\F_q)} [(s), u]^\gamma$, taking $(s)$ to be a $\text{O}_{2m}^\alpha (\F_q)$-conjugacy class
of a semisimple element $s\in \text{SO}_{2m}^\alpha (\F_q)$, a compatible unipotent
representation $u$ of the dual of 
the identity component of $s$'s centralizer $(Z_{\text{SO}_{2m}^\alpha (\F_q)} (s)^\circ)^*$,
and possible extension sign data $\gamma$ consisting of $a (s) +b (s)$ signs.
Write $\gamma = (\alpha, \beta)$, denoting the sign corresponding to possible
$1$ eigenvalues by $\alpha$ and the sign corresponding to possible $-1$ eigenvalues by $\beta$.
Broadly, we obtain new $\text{Sp}_{2N} (\F_q)$-classification data
$$[(\phi(s)),  \phi^\alpha (u), \beta ]$$
by adding $1$ eigenvalues to $s$, then altering the affect factor of the unipotent part
by adding a single coordinate to one row of the symbol (according to $\alpha$ if it occurs
which is precisely when there is a choice), and keeping the original $-1$ part $\beta$ of the sign data if it occurs.

\vspace{3mm}

Consider, again, $s$ as an element of a torus \rref{ToriInAll}.
One may take a direct sum with the identity matrix $I_{2(N-m)+1}$ to obtain a semisimple element
$$\begin{array}{c}
\displaystyle\phi (s) = s \oplus I_{2(N-m)+1} \in \prod_{i=1}^k \text{SO}_{2}^\pm (\F_{q^{n_i}})
\times \text{SO}_{2(N-m)+1} (\F_q) \\
\vspace{1mm} \subseteq \text{SO}_{2N+1} (\F_q) = \text{Sp}_{2N }^* (\F_q).
\end{array}$$
(Note that each different class $(s)$ considered as a conjugacy class in $\text{O} (W,B)$
corresponds to a different $\phi(s)$,
whereas if we only considered $(s)$ as a conjugacy class in $\text{SO}(W,B)$,
in cases with eigenvalues not equal to $\pm 1$, there would be another $\text{SO}(W,B)$-conjugacy class
$(s')$ with $(\phi (s)) = (\phi (s'))$ in $\text{SO}_{2N+1} (\F_q)$.)


Suppose $s$ is of type $(p, \ell)$, i.e. $s$
has $1$ as an eigenvalue of multiplicity $2p$ and $-1$ as an eigenvalue of multiplicity $2\ell$,
and suppose its centralizer is of the form \rref{EvenCentralizerChoices}.
We separate out the factors corresponding to
eigenvalues not equal to $1$ by writing
$$H= \prod_{i=1}^r U_{j_i}^+ (\F_{q^{n_i'}}) \times \prod_{i=1}^t U_{k_i}^- (\F_{q^{n_i''}})
\times \text{SO}_{2 \ell}^\pm(\F_q) $$
so that we have $Z_{\text{SO}_{2m}^\alpha (\F_q)} (s)^\circ = H \times \text{SO}_{2p}^\pm (\F_q),$
and then
\beg{IdComponentZSO2NPhiHSOEv}{Z_{\text{SO}_{2N+1} (\F_q)} (\phi(s))^\circ = H^* \times \text{SO}_{2(N-m+p)+1} (\F_q)}
(note that in this case $H = H^*$).
Factoring $u$ as in \rref{EvenUnipotentFactorization},
let us consider the symbol of the factor corresponding to the $1$ eigenvalue
$u_{\text{SO}_{2p}^\pm}^{+1} = {\lambda_1< \dots < \lambda_a \choose \mu_1< \dots < \mu_b}$
(considering the trivial representation in the case of $p=0$ as ${\emptyset \choose \emptyset}$)
and write $u^{H^*}$
for the representation of $H = H^*$ consisting of the other factors
$\bigotimes_{i=1}^r u_{U_{j_i}^+} \otimes \bigotimes_{i=1}^t u_{U_{k_i}^-}
\otimes u_{\text{SO}_{2\ell}^\pm}^{-1}$.
Suppose first that ${\lambda_1< \dots < \lambda_a \choose \mu_1< \dots < \mu_b}$ is non-degenerate.
In the case where the $1$ eigenvalues correspond to a non-split factor $SO_{2p}^- (\F_q)$ of
the identity component of $s$'s centralizer,
then switch rows so that $a-b+1$ is $1$ mod $4$.
In the case where the $1$ eigenvalues correspond to a split factor $SO_{2p}^+ (\F_q)$ of
the identity component of $s$'s centralizer,
then switch rows so that either $a>b$ or, if $a=b$, for the maximal $i$ such that
$\lambda_i \neq \mu_i$, we have $\lambda_i> \mu_i$.
%
Write
$$N_\pi' = N-m +\frac{a+b}{2}.$$
Then the symbols
${\lambda_1< \dots < \lambda_a < N'_\rho \choose
\mu_1< \dots < \mu_b}$, ${\lambda_1< \dots < \lambda_a \choose
\mu_1< \dots < \mu_b <N'_\rho}$
define unipotent representations of
$\text{SO}_{2(N-m+p)+1}^* (\F_q) = \text{Sp}_{2(N-m+p)} (\F_q)$.
Therefore,
$$\begin{array}{c}
\phi^- (u) = \widetilde{u_{H^*}} \otimes {\lambda_1< \dots < \lambda_a < N'_\pi \choose
\mu_1< \dots < \mu_b}\\[1ex]
\phi^+ (u) = \widetilde{u_{H^*}} \otimes {\lambda_1< \dots < \lambda_a \choose
\mu_1< \dots < \mu_b < N'_\pi}
\end{array}$$
respectively define irreducible unipotent representations of the dual group
\rref{IdComponentZSO2NPhiHSOEv}.
We use $\phi^\alpha (u)$ for the new unipotent data, again writing $\alpha$ for the component
of the $O_{2m}^\pm (\F_q)$-extension sign data corresponding to the $1$ eigenvalues of $s$.

Now suppose
$u_{\text{SO}_{2p}^\pm}^{+1} = {\lambda_1< \dots < \lambda_a \choose
\lambda_1< \dots < \lambda_a}$
is degenerate (counting the case of ${\emptyset\choose \emptyset}$ for $p=0$). These
are precisely the cases where there is no sign corresponding to $1$ eigenvalues
in the $O_{2m}^\pm (\F_q)$-extension sign data. In this case, we put
$$\textstyle \phi (u) = \widetilde{u_H} \otimes {\lambda_1< \dots < \lambda_a <N_\pi' \choose
\lambda_1< \dots < \lambda_a}$$

Finally, to define an output irreducible $\text{Sp}_{2N} (\F_q)$-representation, we 
need to also choose output
central sign data precisely when $\phi (s)$ has $-1$ eigenvalues. By definition, $\phi (s)$ has the same number
of $-1$ eigenvalues as $s$. Therefore, in this case, the original $s$ has $-1$ eigenvalues also,
so the $\text{O}(W,B)$-classification data supplies us with the data of one more central sign
$\beta$, which we use as the output central sign data.

\begin{definition}
Given the above notation, we define $\phi^V_{W,B} (\pi)$ to be the irreducible
$\text{Sp}(V)$-representation with the new classification data we constructed
\beg{EtaDefnSympMetastEven}{\phi^V_{W,B} (r^{\text{O}(W,B)} [(s),u]^\gamma) = 
r^{\text{Sp} (V)} [(\phi (s)), \phi^\alpha (u), \beta ]}
writing $\gamma = (\alpha, \beta)$ for the extension sign data (listing the sign
associated to $1$ eigenvalues before the $-1$ eigenvalues, and omitting either sign
if the corresponding eigenvalue is not present for $s$).
\end{definition}

\subsection{The odd orthogonal stable case}\label{OddOrthoStSubSect}

Suppose $(\text{Sp}(V), \text{O}(W,B))$ is in the orthogonal stable case and $W$ is odd dimensional.
Write $\text{dim} (V) = 2N$ an $\text{dim} (W) = 2m+1$. In
this case, the range condition gives that $m\geq 2N$.

Consider an irreducible representation $\rho=r^{\text{Sp}_{2N} (\F_q)} [(s),u, \pm 1]$ of $\text{Sp}(V)=\text{Sp}_{2N} (\F_q)$ corresponding to the classification data
consisting of a semisimple conjugacy class $(s) \in \text{Sp}_{2N }^*(\F_q) =
\text{SO}_{2N+1} (\F_q)$, a unipotent representation $u$ of
$(Z_{\text{SO}_{2N+1} (\F_q)} (s)^\circ)^*$, and central sign data $\pm 1$ (which we omit
if $-1$ is not an eigenvalue of $s$).
Now our goal is to specify an
irreducible representation $\psi^{W,B}_V (\rho)$ of $\text{O}_{2m+1} (\F_q)$,
by describing $\text{O}_{2m+1} (\F_q)$-extended classification data consisting of
a semisimple conjugacy class $(\psi (s)) \in \text{SO}_{2m+1}^* (\F_q) =\text{Sp}_{2m} (\F_q)$,
a unipotent irreducible representation $\psi^\pm (u)$ of $(Z_{\text{Sp}_{2m} (\F_q)} (\psi (s))^\circ )^*$
(with the superscript $\pm$ specified by the central sign data of the input classification data
and omitting if no such data is given),
and taking $\text{O}_{2m+1} (\F_q)$-extension data corresponding to the sign
$(\epsilon (s) \cdot \text{disc}(B) )$.
Broadly, in this case, we construct $\psi (s)$ by adding $-1$ eigenvalues to $s$
and adding a single new coordinate to the symbol corresponding to the affected factor of
the unipotent part (to achieve
the new needed rank and defect), according the central sign data of $\rho$ if it occurs.

\vspace{3mm}

To be more specific, recall again that we can consider $s$ as an element of a torus of the form \rref{ToriInAll},
by removing the single ``forced"
eigenvalue $1$ from $s$. Write
$\widetilde{s}$ for the $2N$ by $2N$ matrix obtained in this way.
Taking a direct sum with $-I_{2(m-N)}$,
$$\psi (s) := \widetilde{s} \oplus (-I)_{2(m-N)} $$
specifies a semisimple element of $\text{Sp}_{2m} (\F_q) $,
which has $-1$ as an eigenvalue of multiplicity $2(m-N+\ell )$. 

Say that $s$ has $-1$ as an eigenvalue of multiplicity $2\ell$ so that the identity component
of its centralizer is of the form
\rref{SO2lPmInCentForSO2N+1}. Again, let us separate the factors corresponding to eigenvalues
not equal to $-1$ and write
$$H = \prod_{i=1}^r U_{j_i}^+ (\F_{q^{n_i'}}) \times\prod_{i=1}^t U_{k_i}^- (\F_{q^{n_i''}})
\times \text{SO}_{2p+1}(\F_q)$$
so that $Z_{\text{SO}_{2N+1} (\F_q)} (s)^\circ = H \times \text{SO}_{2\ell}^\pm (\F_q)$
and
\beg{IdComponentZSp2mPsiHSOEv}{Z_{\text{Sp}_{2m}(\F_q)} (\psi (s))^\circ = H^* \times \text{Sp}_{2(m-N+\ell)} (\F_q).}

For the unipotent part of the classification data of $\rho$, write its factorization
as in \rref{SymplecticRepsUnipFactors}. Specially consider the symbol corresponding to the factor
$u_{\text{SO}_{2\ell}^\pm } =  {\lambda_1< \dots < \lambda_a \choose
\mu_1< \dots < \mu_b}$
and write $u^{H^*}$ for the unipotent $H^*$-representation corresponding to the
rest of the factors $\bigotimes_{i=1}^r u_{U^+_{j_i}} \otimes\bigotimes_{i=1}^t u_{U^-_{k_i}}\otimes u_{\text{Sp}_{2p}}$.
Again, first suppose the symbol ${\lambda_1< \dots < \lambda_a \choose
\mu_1< \dots < \mu_b}$ is non-degenerate. In the case where the $1$ eigenvalues correspond to a non-split factor $\text{SO}_{2\ell}^- (\F_q)$ of
the identity component of $s$'s centralizer,
then switch rows so that $a-b+1$ is $1$ mod $4$.
In the case where the $1$ eigenvalues correspond to a split factor
$\text{SO}_{2\ell}^+ (\F_q)$ of
the identity component of $s$'s centralizer,
then switch rows so that either $a>b$ or, if $a=b$, for the maximal $i$ such that
$\lambda_i \neq \mu_i$, we have $\lambda_i> \mu_i$.
Let us write
$$m'_\rho = m-N + \frac{a+b}{2}.$$
By the orthogonal stable range condition, we must have
$\lambda_a< m'_\rho$ and $\mu_b< m'_\rho$. Concatenating $m'_\rho$
to the end of one of the rows, we obtain symbols
${\lambda_1< \dots < \lambda_a < m'_\rho\choose
\mu_1< \dots < \mu_b}$, ${\lambda_1< \dots < \lambda_a \choose
\mu_1< \dots < \mu_b < m'_\rho}$
which have odd defect and rank precisely equal to $m-N+\ell$, and therefore specify
irreducible unipotent representations of $\text{SO}_{2(m-N+\ell)+1} (\F_q)
= \text{Sp}_{2(m-N+\ell)}^* (\F_q)$.
Putting
$$\begin{array}{c}
\psi^- (u) = \widetilde{u_{H^*}} \otimes {\lambda_1< \dots < \lambda_a < m'_\rho\choose
\mu_1< \dots < \mu_b}\\[1ex]
\psi^+ (u) = \widetilde{u_{H^*}} \otimes {\lambda_1< \dots < \lambda_a \choose
\mu_1< \dots < \mu_b < m'_\rho}
\end{array}$$
we obtain unipotent representations of \rref{IdComponentZSp2mPsiHSOEv}.

Suppose now that $u_{\text{SO}_{2\ell}^\pm} = {\lambda_1< \dots < \lambda_a \choose \lambda_1<\dots<\lambda_a}$ is degenerate (including the case of the trivial representation
${\emptyset\choose \emptyset}$ for $\ell = 0$). This is precisely the case where no central sign
data is provided in $\rho$'s original $\text{Sp} (V)$-classification data. In this case put
$$\textstyle \psi (u) = \widetilde{u_{H^*}} \otimes {\lambda_1< \dots < \lambda_a < m_\rho' 
\choose \lambda_1<\dots < \lambda_a}.$$


\begin{definition}
Suppose we are given the above notation.
We put 
\beg{ell=0orthoodd}{\psi^{W,B}_V (r^{\text{\em Sp} (V)} [(s),u]) = r^{\text{\em O}(W,B)}[(\psi (s)), \psi (u)]^{ \epsilon (s) \cdot \text{disc} (B)} }
or
\beg{ell>0orthoodd}{\psi^{W,B}_V (r^{\text{\em Sp} (V)} [(s),u, \pm 1]) = r^{\text{\em O}(W,B)}[(\psi (s)), \psi^\pm (u)]^{ \epsilon (s) \cdot \text{disc} (B)}.}
\end{definition}

\subsection{The even orthogonal stable case}\label{EvenOrthoStSubSect}
Suppose $(\text{Sp} (V), \text{O}(W,B))$ is in the orthogonal stable range
and $W$ is of even dimension $2m$,
Write $\text{dim} (V) = 2N$ and
fix the sign $\pm$ so that $\text{O}(W,B) =\text{O}^\pm_{2m} (\F_q)$.
Then the range condition gives $m \geq 2N$ in the split case 
$\text{O}(W,B) =\text{O}^+_{2m} (\F_q)$ and $m-1 \geq 2N$ in the non-split case
$\text{O}(W,B) =\text{O}^-_{2m} (\F_q)$.

Again, our goal is to produce a construction
with input an irreducible representation $\rho=r^{\text{Sp}_{2N} (\F_q)} [(s),u, \pm 1]$ of $\text{Sp}(V)=\text{Sp}_{2N} (\F_q)$ corresponding to the classification data
consisting of a semisimple conjugacy class $s \in \text{Sp}_{2N }^*(\F_q) =
\text{SO}_{2N+1} (\F_q)$, a unipotent representation $u$ of
$(Z_{\text{SO}_{2N+1} (\F_q)} (s)^\circ)^*$, and central sign data $\pm 1$ (which we omit
if $-1$ is not an eigenvalue of $s$).
We want to produce an irreducible $\text{O}_{2m}^\pm (\F_q)$-representation $\psi^{W,B}_V (\rho)$
which is defined by extended classification data
consisting of a semisimple $\text{O}_{2m}^\pm (\F_q)$-conjugacy class $(\psi(s)) \in \text{SO}_{2m}^\pm (\F_q)$, a unipotent 
representation $\psi(u)$ of 
the dual of the 
the identity component
of the centralizer of $\psi (s)$ in $\text{SO}_{2m}^\alpha (\F_q)$, and extension sign data depending on
whether $\psi (s)$ has $\pm 1$ eigenvalues.
Broadly, we will construct the new semisimple and unipotent parts of the $\text{O}_{2m}^\pm (\F_q)$-classification data
$$[(\psi (s)), \psi (u)]$$
by adding $1$ eigenvalues to $s$ and altering the symbol of the affect factor of the unipotent part
by adding a single new coordinate to one of the rows to achieve the new needed rank and defect.

To be more specific, 
as in the odd stable orthogonal case, we may remove a single ``forced" $1$ eigenvalue from $s$
to view it as a $2N$ by $2N$
element $\widetilde{s}$ of a maximal torus \rref{ToriInAll}.
Then consider the direct sum with the unique appropriate choice of
$2(m-N)$ by $2(m-N)$ identity matrix
$$\psi (s) = \widetilde{s} \oplus I_{2(m-N)},$$
configured to give
a $2m$ by $2m$ matrix that can be considered as an element of $\text{SO}(W,B) \subseteq \text{O}(W,B)$.
As in Subsection \ref{EvenSympStSubSect}, each distinct
$\text{SO}_{2N+1}(\F_q)$-conjugacy class
$(s)$ gives a distinct $\text{O}_{2m}^\pm (\F_q)$-conjugacy class $\psi (s)$.
Writing the identity component of the centralizer of $s$ as \rref{SO2lPmInCentForSO2N+1}, we again separate out
the factors corresponding to the eigenvalues not equal to $1$, writing
$$H =  \prod_{i=1}^r U_{j_i}^+ (\F_{q^{n_i'}}) \times\prod_{i=1}^t U_{k_i}^- (\F_{q^{n_i''}}) \times 
\text{SO}_{2\ell}^\pm (\F_q)$$
so that
$Z_{\text{SO}_{2N+1} (\F_q)} (s)^\circ = H \times \text{SO}_{2p+1} (\F_q)$
and
\beg{IdComponentZSp2mPsiHSOSigned}{Z_{\text{SO}_{2m}^\pm (\F_q)} (\psi (s))^\circ = H^* \times \text{SO}_{2(N-m+p)}^\beta (\F_q),}
for a single determined choice of sign $\beta$
(so that its product with the other signs appearing in $H$ agrees
with $\alpha$). Note that $H=H^*$ in this case.

To construct the unipotent part of the $\text{O}_{2m}^\pm (\F_q)$-classification data $\psi (u)$,
specially consider the symbol
$u_{\text{SO}_{2p+1}} = {\lambda_1 < \dots < \lambda_a \choose
\mu_1< \dots < \mu_b}$
and write $u^{H^*}$ for the product of the remaining factors 
$\bigotimes_{i=1}^r u_{U^+_{j_i}} \otimes\bigotimes_{i=1}^t u_{U^-_{k_i}} \otimes u_{\text{SO}_{2\ell}^\pm} $.
Switch the symbol rows so that the defect $a-b$ is $1$ mod $4$ (which is
possible since this symbol has odd defect). Let us write
$$m'_\rho = m-N +\frac{a+b-1}{2}.$$
Then, if $\beta = +$, if $\mu_b < m'_\rho$, putting
$$\textstyle \psi (u) = \widetilde{u_{H^*}} \otimes {\lambda_1< \dots < \lambda_a \choose
\mu_1< \dots < \mu_b < m'_\rho}$$
gives a unipotent representation of the group \rref{IdComponentZSp2mPsiHSOSigned}.
Similarly, if $\beta = -$, if $\lambda_a < m'_\rho$, putting
$$\textstyle \psi (u) = \widetilde{u_{H^*}} \otimes {\lambda_1< \dots < \lambda_a < \mu_\rho'\choose
\mu_1< \dots < \mu_b }$$
gives a unipotent representation of the group \rref{IdComponentZSp2mPsiHSOSigned}.

Now for the extension data, in the stable range $\psi (s)$ always has $1$ eignevalues
and has the same multiplicity of
$-1$ eigenvalues as the input semisimple element $s$.
First, we put the sign corresponding to the $1$ eigenvalues always to be $+1$
(when $(\text{Sp} (V), \text{O}(W,B))$ lies in the stable range).
When $\psi (s)$ has $-1$ eigenvalues, we take the sign data to match the central sign
data of the input $\text{Sp} (V)$-classification data.

\begin{definition}
Suppose we are given the above notation. We put
\beg{ZetaDefnOrthoEven}{\psi^{W,B}_V (r^{\text{\em Sp} (V)} [(s), u] ) := r^{\text{\em O}(W,B)} [(\psi (s)), \psi (u)]^{(+1)}}
(where the single piece of extension sign data corresponds to $1$ eigenvalues of $\psi (s)$) and
\beg{ZetaDefnOrthoEvenSign}{\psi^{W,B}_V (r^{\text{\em Sp} (V)} [(s), u, \pm 1] ) := r^{\text{\em O}(W,B)} [(\psi (s)), \psi (u)]^{(+1, \pm 1)}.}

\end{definition}

\section{A combinatorial identity}\label{CombinatoricsSect}

Now recalling \cite{TotalHoweI}, a key step in decomposition the restriction
of an oscillator representation $\omega [ V\otimes W]$ to $\text{Sp}(V)\times \text{O}(W,B)$
is to separate off its ``top part," which specifically singles out summands
arising from the eta or zeta correspondence with source corresponding to the
appropriate full-rank orthogonal or symplectic group, respectively. 
In the symplectic stable range, we write
$${ \omega} [V \otimes W]^{\text{top}} := \bigoplus_{\rho \in \text{O} (W,B)} 
\eta^V_{W,B} (\rho) \otimes \rho,$$
and call it the {\em top part} of $ \omega [V\otimes W]$.
Similarly, in the orthogonal stable range,
we write
$${ \omega} [V \otimes W]^{\text{top}'} := \bigoplus_{\rho \in \text{O} (W,B)} 
\rho \otimes \zeta^{W,B}_V (\rho),$$
and call it the {\em top part} of $ \omega [V \otimes W]$.

From here, the proof of Theorem \ref{EtaExplicitIntro} separates into two key steps: A combinatorial
verification that the dimension of the direct sum of 
matches the dimension of the top part of $ \omega [ V \otimes W]$,
and an inductive
argument showing that the claimed correspondence in Theorem \ref{EtaExplicitIntro} is the only possible
one. The first step is the goal of this section.

\begin{theorem}\label{DimsMatchTheoremEta}
If $(\text{Sp}(V), \text{O}(W,B))$ is in the symplectic stable range,
the dimension of the top part of the restriction of $\omega [V\otimes W]$ matches the
sum of products of the dimensions of irreducible representations of $\text{O}(W,B)$ and their
$\psi^V_{W,B}$ correspondences:
\beg{DimsForEtaIsPhi}{\text{dim} ({ \omega} [V \otimes W]^{\text{top}}) = \sum_{\rho \in \widehat{\text{O} (W,B)}}
\text{dim} (\rho) \cdot \text{dim} (\phi^V_{W,B} (\rho)).}
Similarly, if $(\text{Sp}(V), \text{O}(W,B))$ is in the orthogonal stable range,
the dimension of the top part of the restriction of $\omega [V\otimes W]$ matches the
sum of products of the dimensions of irreducible representations of $\text{Sp}(V)$ and their
$\phi^{W,B}_V$ correspondences:
\beg{DimsForZetaIsPsi}{\text{dim} ({ \omega} [V \otimes W]^{\text{top}'}) = \sum_{\rho \in \widehat{\text{Sp}(V)}}
\text{dim} (\rho) \cdot \text{dim} (\psi^{W,B}_V (\rho)).}
\end{theorem}

\subsection{The dimension of the top part of the oscillator representation}
First, we need a more explicit formula for the left hand side of \rref{DimsForEtaIsPhi}:

\begin{proposition}\label{SympTopPartProp}
Consider symplectic and orthogonal spaces $V$ and $(W,B)$ whose dimensions
are in the symplectic stable range.
Writing $\text{dim} (V) = 2N$ and $\text{dim} (W) = 2m+1$,
the dimension of the top part of the restriction of $ \omega [ V \otimes W]$ is
\beg{TopDimFormula}{ \sum_{i=0}^m (-1)^{m-i}  q^{m-i \choose 2}  {m \choose i}_{\hspace{-1mm}q}   \prod_{k= i+1}^m (q^k +1)  q^{(2i + 1)N}}
\end{proposition}

\vspace{3mm}

\begin{proof}
Write, for $j< i$
\beg{CCoefficientDefn}{\begin{array}{c}
\displaystyle C_{i,j} := - {i \choose j}_{\hspace{-1mm}q}   \prod_{k=j+1}^i (q^k +1)= \\
\vspace{-1.5mm}\\
\displaystyle \frac{(q^{2i} -1) (q^{2(i-1)}-1) \dots (q^{2(j+1)}-1)}{(q^{i-j} -1) (q^{i-j-1} -1) \dots (q-1)}\end{array}}
Let $X_m$ denote the dimension of the top part $ \omega [V \otimes W]^{\text{ top}}$,
where $W$ is a $2m+1$-dimensional $\F_q$-space.
Taking the dimension of \rref{EtaCorrThmDecomp} then gives the recursive equation
\beg{IterativeFormulaTopDim}{X_m = q^{(2m+1)N} + \sum_{i=0}^{m-1} C_{m, i}  X_i .}

Our goal to prove \rref{TopDimFormula} is to re-express the right-hand side of
\rref{IterativeFormulaTopDim} in terms of a sum of $q^{(2i+1)N}$ for $0 \leq i \leq m$ some lower
coefficient.
Now, iteratively applying \rref{IterativeFormulaTopDim}, we find that
$$X_m = \sum_{i=0}^m \left( \sum_{i = \ell_1 < \dots < \ell_j = m} \hspace{1mm} \prod_{k=1}^{j-1} C_{\ell_{k+1}, \ell_k} \right) q^{(2i+1)N}.$$

It suffices to prove
\beg{TopDimCsLastStep}{ \sum_{i = \ell_1 < \dots < \ell_j = m} \hspace{1mm} \prod_{k=1}^{j-1} C_{\ell_{k+1}, \ell_k}  = C_{m, i} \; q^{i \choose 2}.}
Using \rref{CCoefficientDefn}, each term $\prod_{k=1}^{j-1} C_{\ell_{k+1}, \ell_k}$
where $i = \ell_1 < \dots < \ell_j = m$,
factors as
$$\frac{(q^{2m} -1) (q^{2(m-1)} -1) \dots ( q^{2(m-i+1)} -1)}{\displaystyle \prod_{k=1}^{j-1}\; \prod_{r= 1}^{\ell_{k+1} - \ell_k} (q^r -1 )},
$$
which can be simplified as
$$
C_{m, i}  \frac{(q^{m-i} -1) (q^{m-i-1} -1) \dots (q-1)}{\displaystyle \prod_{k=1}^{j-1}\; \prod_{r= 1}^{\ell_{k+1} - \ell_k} (q^r -1)},
$$
reducing the claim to
\beg{QMultinomThmClaim}{q^{m-i \choose 2} = \sum_{i = \ell_1< \dots < \ell_j = m} \frac{(q^{m-i} -1) (q^{m-i-1} -1) \dots (q-1)}{\displaystyle \prod_{k=1}^{j-1}\; \prod_{r= 1}^{\ell_{k+1} - \ell_k} (q^r -1)}.}
The right-hand side of \rref{QMultinomThmClaim} can also be written as
$$
\sum_{0= \ell_1' < \dots < \ell_j' = m-i} {\ell_j' \choose \ell_{j-1}'}_{\hspace{-1mm}q} { \ell_{j-1}' \choose \ell_{j-2}' }_{\hspace{-1mm}q} \dots {
\ell_2' \choose \ell_1'}_{\hspace{-1mm}q},
$$
by substituting $\ell_j' = \ell_j -i$, so \rref{QMultinomThmClaim} follows from a $q$-version of the multinomial theorem.
\end{proof}

We re-write \rref{TopDimFormula} again as follows, to separate
it into terms which correspond to levels of singularity of semisimple elements
(more specifically, the multiplicity of eigenvalue $-1$) in the classification
of irreducible representations of $\text{SO}_{2m+1} (\F_q)$:
\begin{proposition}\label{CombinEleOdd}
The top dimension of $ \omega [V\otimes W]^{top}$ is 
\beg{Step2Expression}{
\sum_{\ell = 0}^m (-1)^\ell \; q^{N+ (m-\ell) (m-\ell-1) + \ell^2} \; {m \choose \ell}_{\hspace{-1mm}q^2}
 \prod_{j=0}^{m-\ell-1} (q^{2(N-i)}-1).
}
\end{proposition}

\begin{proof}
Substituting $i = m-\ell$, \rref{TopDimFormula}
can be re-written as
\beg{Substituim-elltopcomb}{\sum_{\ell= 0}^m (-1)^\ell q^{{\ell \choose 2}} 
{m \choose \ell}_{\hspace{-1mm}q^2} (\prod_{j=1}^{\ell} (q^j +1) )q^{(2(m-\ell)+1)+N}.}

Now in \rref{Step2Expression}, using
$$(m-\ell-1) (m-\ell) = \sum_{j=0}^{m-\ell-1} 2k,$$
we have
$$q^{(m-\ell-1)(m-\ell)} \prod_{j=0}^{m-\ell -1} (q^{2(N-i)} -1)  = \prod_{j=0}^{m-\ell-1}
(q^{2N} - q^{2i}).$$
Hence, \rref{Step2Expression} reduces to
$$\sum_{\ell= 0}^m (-1)^\ell q^{N+\ell^2} {m \choose \ell}_{q^2}\prod_{i=0}^{m-\ell -1} 
(q^{2N} - q^{2i}).$$
Finally, at each $\ell$,
$$
\begin{array}{c}
\displaystyle \prod_{i= 0}^{m-\ell-1} (q^{2N} -q^{2i}) = \sum_{j=0}^{m-\ell} q^{2Nj} \sum_{1\leq i_1< \dots < i_{m-\ell-j} \leq m-\ell-1} q^{2(i_1 + \dots +i_{m-\ell-j})}=\\
\\
\displaystyle \sum_{j=0}^{m-\ell} q^{2Nj} {m-\ell \choose j}_{\hspace{-1mm}q^2}.
\end{array}$$
Therefore, the coefficient of $q^{(2(m-\ell)+1)N}$ in \rref{Substituim-elltopcomb} for
each $\ell$ is
$$\sum_{k= 0}^{\ell} (-1)^k q^{k^2} {m \choose m-k}_{\hspace{-1mm}q^2} {m-k \choose m-\ell}_{\hspace{-1mm}q^2}$$
(identifying ${m \choose k}_{\hspace{-1mm}q^2}$ with ${m \choose m-k}_{\hspace{-1mm}q^2}$).
Hence, the claim reduces to verifying that
\beg{LastQBinnom}{
\begin{array}{c}
\displaystyle \sum_{k=0}^\ell (-1)^k q^{k^2} {m \choose m-k}_{\hspace{-1mm}q^2} {m-k \choose m-\ell}_{\hspace{-1mm}q^2}=\\
\\
\displaystyle q^{\ell \choose 2} {m \choose m-\ell}_{\hspace{-1mm}q^2} \prod_{j=1}^\ell (q^j +1)
\end{array}}
Further, we have
$$
{m \choose m-k}_{q^2} {m-k \choose m-\ell}_{q^2} = {m \choose m-\ell}_{q^2} {\ell \choose \ell-k}_{q^2}= {m \choose m-\ell}_{q^2} {\ell \choose k}_{q^2},
$$
reducing \rref{LastQBinnom} again to a $q$-multinomial theorem.
\end{proof}

The purpose of re-writing the dimension of the top part of $ \omega [ V\otimes W]$
as \rref{Step2Expression} is because, for each $\ell$, the prime to $q$
part of the $\ell$th term of \rref{Step2Expression} is
\beg{WhyStep2Express}{\begin{array}{c}
\displaystyle {m \choose \ell}_{\hspace{-1mm}q^2} \; \prod_{i=0}^{m-\ell-1} (q^{2(N-i)} -1) =\\
\\
\displaystyle \frac{|\text{Sp}_{2m} (\F_q)|_{q'} \cdot |\text{Sp}_{2N} (\F_q)|_{q'}}{|\text{Sp}_{2(m-\ell)} (\F_q)|_{q'} \cdot |\text{Sp}_{2\ell} (\F_q)|_{q'} \cdot |\text{Sp}_{2(N-m-\ell)}(\F_q) |_{q'}.}
\end{array}
}

We use Proposition \ref{CombinEleOdd}
to conclude \rref{DimsForEtaIsPhi} by approximating the right hand recursively by considering
terms $\text{dim} (\pi) \text{dim} (\phi_{W,B} (\pi))$ separately for $\pi \in \widehat{\text{O}( W,B)}$ arising from a
conjugacy class of a semisimple element of the dual group $\text{Sp}_{2m} (\F_q)$,
which is singular of type $(m-\ell, \ell)$ (i.e. has $-1$ as an eigenvalue with multiplicity $2\ell$), using
the elementary fact that the sum of the squares of the dimensions of all irreducible
representations of a group $G$ recover its group order.
This gives that the ``level $\ell$" approximation of the right hand side of \rref{DimsForEtaIsPhi}
(which counts correctly the terms from $\pi$ arising from conjugacy classes of semisimple
elements $\text{Sp}_{2m} (\F_q)$ with eigenvalue $-1$ of multiplicity less than or equal to $2\ell$,
and miss-counts the terms from $\pi$ arising from conjugacy classes with eigenvalue $-1$ of multiplicity
more than $2\ell$) is the sum of the first $\ell$ terms of \rref{Step2Expression}.

More formally:
\begin{proof}[Proof of Theorem \ref{DimsMatchTheoremEta}]
Suppose $W$ is an $\F_q$-vector space of dimension $2m+1$ with symmetric bilinear form $B$.
First, consider irreducible representations $\pi \in \widehat{\text{O}(W,B)}$
whose restrictions $\text{Res}_{\text{SO}(W,B)} (\pi)$ to $\text{SO}(W,B)$ correspond
to a conjugacy class of a semisimple element $s \in \text{Sp}_{2m} (\F_q)$
where $-1$ is not an eigenvalue. Call such irreducible representations
of $\text{SO}(W,B)$ the {\em ``level $0$" representations
of $\text{SO}(W,B)$}. For each such $\pi' \in \widehat{\text{SO}(W,B)}$, say with classification
given by the Jordan decomposition $[(s), u]$, the identity component of the centralizer of $s$ first must be of the form
$$Z_{\text{Sp}_{2m } (\F_q)} (s)^\circ = \prod_{i=1}^r U_{j_i}^+ (\F_q) \times \prod_{i=1}^t U_{k_i}^-
(\F_q) \times \text{Sp}_{2p} (\F_q)$$
for $\sum_{i=1}^r j_i + \sum_{i=1}^t k_i +p = m$ (where $s$ has $1$ as an eigenvalue
with multiplicity $2p$), with the unipotent representation $u$ then consisting of the data of unipotent
representations of $U_{j_i}^+ (\F_q)$, $U_{k_i}^- (\F_q)$, and a symbol of rank $p$ and type $C$.
Then,
$$\begin{array}{c}
Z_{\text{SO}_{2N+1} (\F_q)} (s \oplus \sigma_{N-m}^\pm )^\circ
= \\[1ex]
\displaystyle \prod_{i=1}^r U_{j_i}^+ (\F_q) \times \prod_{i=1}^t U_{k_i}^- (\F_q) \times \text{SO}_{2p+1} (\F_q)
 \times \text{SO}_{2(N-m)}^\pm (\F_q)
\end{array},$$
wich has order
$$|Z_{\text{SO}_{2N+1} (\F_q)} (s \oplus \sigma_{N-m}^\pm )^\circ | = | Z_{\text{Sp}_{2m} (\F_q)}(s)^\circ | \cdot
|\text{SO}_{2(N-m)}^{\pm disc(B)} (\F_q)|.$$
For both choices, the dimension of $\phi_{W,B} (u)$ is equal to the dimension of $u$.
Hence, for every $\pi' \in \widehat{\text{SO}(W,B)}$,
the sum of dimensions $\text{dim} (\phi_{W,B} (\pi \otimes 1)) + \text{dim}(\phi_{W,B} (\pi \otimes -1))$
is equal to the dimension of $\pi$, multiplied by
$$
\begin{array}{c}
\displaystyle \frac{|\text{Sp}_{2N} (\F_q)|_{q'}}{2|\text{SO}_{2(N-m)}^+ (\F_q) \times \text{Sp}_{2m} (\F_q)|_{q'}} 
+\frac{|\text{Sp}_{2N} (\F_q)|_{q'}}{2|\text{SO}_{2(N-m)}^- (\F_q) \times \text{Sp}_{2m }(\F_q)|_{q'}} =\\
\\
\displaystyle \frac{1}{|\text{Sp}_{2m }(\F_q)|_{q'}} q^{N-m} \prod_{i=0}^{m-1} (q^{2(N-i)} -1) 
\end{array}$$
Hence, since the dimensions of the $\text{O}(W,B)$ representations $\pi \otimes 1$ and $\pi \otimes -1$
are equal to $\text{dim} (\pi)$, the sum of the two terms
\beg{Level0CombinatorialFinal}{\begin{array}{c}
\displaystyle \text{dim} (\pi' \otimes 1 ) \text{dim} (\phi_{W,B} (\pi' \otimes 1)) + \text{dim} (\pi '\otimes -1) \text{dim} (\phi_{W,B} (
\pi' \otimes -1)) = \\
\\
\displaystyle \frac{\text{dim} (\pi')^2}{|\text{Sp}_{2m }(\F_q)|_{q'}} q^{N-m} \prod_{i=0}^{m-1} (q^{2(N-i)} -1)
\end{array}}
If all representations $\pi' \in \widehat{\text{SO}(W,B)}$ satisfied \rref{Level0CombinatorialFinal},
then the right hand side of \rref{DimsForEtaIsPhi} would equal
\beg{Level0ApproximationFinalCombArg}{\begin{array}{c}
\displaystyle\frac{|\text{SO} (W,B) |}{|\text{Sp}_{2m} (\F_q)|_{q'}} q^{N-m} \prod_{i=0}^{m-1} (q^{2(N-i)} -1)=\\
\\
\displaystyle q^{N-m (m-1)} \prod_{i=0}^{m-1} (q^{2(N-i)} -1)
\end{array}}
(recalling that $|\text{SO} (W,B) | = |\text{Sp}_{2m} (\F_q)|$, with $q$-part equal to $q^{m^2}$).

We call \rref{Level0ApproximationFinalCombArg} the {\em level 0 approximation}
of \rref{DimsForEtaIsPhi}. Note that it is precisely equal to the $0$th term of \rref{DimsForEtaIsPhi}.
The remainder of the argument consists of considering
the ranges of irreducible representations arising from semisimple elements
one $\ell$ at a time (from $\ell = 1$ to $\ell = m$).
We must compute that adding the $\ell$th term of \rref{Step2Expression}
cancels the ``level $\ell$-error,"
arising from miscounting the terms \rref{DimsForEtaIsPhi} 
for $\pi'$ arising from $(s)$ with exactly $2\ell$
in the level $(\ell -1)$-approximation of \rref{DimsForEtaIsPhi} (though it may create
more error at higher levels), so that we can take the sum
$$
\sum_{i=0}^{\ell} (-1)^i q^{N+ (m- i)(m- i -1) + i ^2} {m \choose \ell}_{q^2} \prod_{j = 0}^{m- i -1}
(q^{2(N-j)} -1) 
$$
to be the {\em level $\ell$ approximation} of \rref{DimsForEtaIsPhi}.

\vspace{3mm}

We may therefore prove Theorem \ref{DimsMatchTheoremEta}
inductively by verifying that for every $\ell$, the level $\ell$-approximation
is equal to the sum of the $0$th to $\ell$th terms of \rref{Step2Expression},
up to an error of terms with $N$-degree
less than or equal to $2(m-\ell ) +1$.

\begin{lemma}
Fix a symbol $\lambda_1 < \dots < \lambda_a \choose \mu_1 < \dots < \mu_b$
of rank $\ell$, type $C$, and write $c= (a+b-1)/2$.
Choosing the sign of $\text{SO}_{2N}^\pm (\F_q)$ in the denominator according to matching
the defect of the written symbols with the appropriate groups (depending on $a-b$ mod $4$),
then the sum
\beg{SumLemma}{\begin{array}{c}
\displaystyle \frac{|\text{Sp}_{2N} (\F_q)|_{q'}}{|\text{SO}_{2N}^\pm (\F_q)|_{q'}} \text{dim} (
{\lambda_1< \dots < \lambda_a \choose \mu_1< \dots < \mu_b < N-\ell +c})+\\
\\
\displaystyle \frac{|\text{Sp}_{2N} (\F_q)|_{q'}}{|\text{SO}_{2N}^\mp (\F_q)|_{q'}} \text{dim} (
{\lambda_1< \dots < \lambda_a<N-\ell+c \choose \mu_1< \dots < \mu_b}) 
\end{array}
}
is the product
\beg{MainCancellableTerm}{
\begin{array}{c}
\displaystyle (\frac{|\text{Sp}_{2N} (\F_q)|_{q'}}{|\text{SO}_{2\ell+1} (\F_q) \times 
\text{SO}_{2(N-\ell)}^+ (\F_q) |_{q'}} + \frac{|\text{Sp}_{2N} (\F_q)|_{q'} }{
|\text{SO}_{2\ell+1} (\F_q) \times \text{SO}_{2(N-\ell)}^- (\F_q) })\cdot\\
\\
\displaystyle \text{dim} ({\lambda_1< \dots < \lambda_a \choose
\mu_1 < \dots < \mu_b}),
\end{array}
}
up to an error term equal to \rref{MainCancellableTerm}, multiplied again
by $\text{dim} ({\lambda_1< \dots < \lambda_a \choose \mu_1< \dots < \mu_b})$
and the factor
$$q^{N-(m-\ell)(m-\ell -1)} \frac{|\text{Sp}_{2N} (\F_q)|_{q'}}{|\text{Sp}_{2(N-m+\ell)} (\F_q)|_{q'}}\frac{|\text{SO}_{2m+1} (\F_q)|_{q'}}{|\text{SO}_{2\ell+1} (\F_q) \times \text{SO}_{2(m-\ell)+1} (\F_q)|_{q'}|}
$$
\end{lemma}

\begin{proof}
Suppose, without loss of generality, $a-b$ is $1$ mod $4$.
Recalling how to compute the dimensions of symbols, we have that
$$\begin{array}{c}
\displaystyle \frac{\displaystyle \text{dim}(
{\lambda_1< \dots < \lambda_a \choose \mu_1< \dots < \mu_b<N-\ell+c})
}{
q^{ {a+b-1 \choose 2} + {a+b-3 \choose 2} + {a+b-5 \choose 2} + \dots }\cdot |\text{SO}_{2N}^+ (\F_q)|_{q'}} = \\
\\
\displaystyle
\frac{\displaystyle \text{dim} ({\lambda_1< \dots < \lambda_a
\choose \mu_1< \dots < \mu_b})  \prod_{i=1}^a (q^{N-\ell + c} + q^{\lambda_i}) \prod_{i=1}^b (q^{N-\ell +c}- q^{\mu_i}) }{
\displaystyle q^{  {a+b-2 \choose 2} + {a+b-4 \choose 2} + {a+b-6 \choose 2}+  \dots} \cdot |\text{SO}_{2\ell+1} (\F_q)|_{q'} \prod_{i=1}^{N-\ell+c} (q^{2i}-1)}
\end{array}
$$
and, similarly,
$$\begin{array}{c}
\displaystyle
\frac{\displaystyle \text{dim}(
{\lambda_1< \dots < \lambda_a < N-\ell +c \choose \mu_1< \dots < \mu_b })
}{
q^{ {a+b-1 \choose 2} + {a+b-3 \choose 2} + {a+b-5 \choose 2} + \dots }\cdot |\text{SO}_{2N}^- (\F_q)|_{q'}} = \\
\\
\displaystyle
\frac{\displaystyle \text{dim} ({\lambda_1< \dots < \lambda_a
\choose \mu_1< \dots < \mu_b})  \prod_{i=1}^a (q^{N-\ell + c} - q^{\lambda_i}) \prod_{i=1}^b (q^{N-\ell +c} + q^{\mu_i}) }{
\displaystyle 
q^{ {a+b-2 \choose 2} + {a+b-4 \choose 2} + {a+b-6 \choose 2} + \dots }\cdot|\text{SO}_{2\ell +1} (\F_q)|_{q'} \prod_{i=1}^{N-\ell+c} (q^{2i}-1)}.
\end{array}
$$
Summing the terms \rref{SumLemma} then gives a product of the coefficient
$$\frac{|\text{Sp}_{2N} (\F_q)|_{q'}}{\displaystyle |\text{SO}_{2\ell+1} (\F_q)|_{q'} \prod_{i=1}^{N-\ell+c} (q^{2i} -1)} \text{dim} ({\lambda_1< \dots<\lambda_a
\choose \mu_1< \dots < \mu_b}),$$
with the factor
$$\frac{q^{  {a+b-2 \choose 2} + {a+b-4 \choose 2} + {a+b-6 \choose 2}+  \dots }}{
q^{ {a+b-1 \choose 2} + {a+b-3 \choose 2} + {a+b-5 \choose 2} + \dots }} =
\frac{1}{q^{ \sum_{i=0}^{c-1} (2(c-i)-1)}}
=\frac{1}{q^{c^2}},$$
with the sum
\beg{TotalSwapCOmb}{\begin{array}{c}
\displaystyle \prod_{i=1}^a (q^{N-\ell+c} - q^{\lambda_i}) \prod_{i=1}^b (q^{N-\ell+c}+ q^{\mu_i})
+\\
\\
\displaystyle \prod_{i=1}^a (q^{N-\ell+c} +q^{\lambda_i}) \prod_{i=1}^b (q^{N-\ell+c} - q^{\mu_i}).
\end{array}
}
Since the defect $a-b$ is odd, when multiplying out the factors \rref{TotalSwapCOmb}
as a sum of powers of $q^{N-\ell +c}$ (with lesser coefficients, not involving
$N$), we find that only odd powers $q^{(2k+1)(N-\ell +c)}$ have non-zero coefficient
(for $k = 0, \dots , c$).
Explicitly, it is
$$2 q^{N-\ell+c} (\sum_{k=0}^c q^{2k (N-\ell+c)} \cdot \sum (-1)^{r} q^{\sum_{s=1}^r \lambda_{i_s}
+ \sum_{s=1}^{2(c-k)-r} \mu_{j_s}} )$$
where the second sum runs over all choices of $r$ 
and $1\leq i_1< \dots < i_r \leq a$, $1 \leq j_1< \dots < j_{2(c-k)-r} \leq b$.
Consider
$$\begin{array}{c}
2 q^{N-\ell +c} = q^c ((q^{N-\ell} -1) + (q^{N-\ell} +1))
= \\
\\
\displaystyle 
q^c ( \frac{| \text{Sp}_{2(N-\ell)} (\F_q)|_{q'}}{|\text{SO}_{2(N-\ell)}^+ (\F_q)|_{q'}} +
\frac{|\text{Sp}_{2(N-\ell)} (\F_q)|_{q'}}{|\text{SO}_{2(N-\ell)}^- (\F_q)|_{q'}}).
\end{array}$$

\vspace{3mm}

Redistributing terms, this can be re-expressed as the product of
$$\begin{array}{c}
\displaystyle (\frac{|\text{Sp}_{2N} (\F_q)|_{q'}}{|\text{SO}_{2\ell+1} (\F_q) \times \text{SO}_{2(N-\ell)}^+ (\F_q)|_{q'}} + 
\frac{|\text{Sp}_{2N} (\F_q)|_{q'}}{|\text{SO}_{2\ell+1} (\F_q) \times \text{SO}_{2(N-\ell)}^- (\F_q)|_{q'}}) \cdot\\
\\
\displaystyle \text{dim} ({\lambda_1< \dots < \lambda_a \choose \mu_1 < \dots < \mu_b})
\end{array}$$
with the fraction
\beg{OppositeFactor}{\begin{array}{c}
\displaystyle
\frac{\displaystyle q^c \cdot \sum_{k=0}^c q^{2k(N-\ell+c)}\cdot  \sum (-1)^r q^{\sum_{s=1}^r \lambda_{i_s} +
\sum_{s=1}^{2(c-k)-r} \mu_{j_s}} }{ \displaystyle q^{c^2} \cdot
\prod_{i=1}^c (q^{2(N-\ell+i)}-1) } = \\
\\
\displaystyle
\frac{\displaystyle \sum_{k=0}^c q^{2k(N-\ell+c)}\cdot  \sum (-1)^r q^{\sum_{s=1}^r \lambda_{i_s} +
\sum_{s=1}^{2(c-k)-r} \mu_{j_s}} }{\displaystyle \prod_{i=0}^{c-1} (q^{2(N-\ell+c)} -q^{2i})}.
\end{array}.}
In particular, the top degree of $q$ in both
the numerator and demoniator of \rref{OppositeFactor}
is $2c(N-\ell+c)$.
Finally, therefore \rref{OppositeFactor} reduces as $1$ (contributing the claimed main term),
summed with
$$\frac{\displaystyle 
\sum_{k=0}^{c-1} q^{2k (N-\ell+c)} \cdot (\sum  (-1)^r q^{\sum_{s=1}^r \lambda_{i_s}
+ \sum_{s=1}^{2(c-k)-r} \mu_{j_s}} - { c\choose k}_{q^2})
}{\displaystyle \prod_{i=0}^{c-1} (q^{2(N-\ell+c)} -q^{2i})},$$
recalling
$$\sum_{0 \leq \ell_1 < \dots \ell_{c-k} \leq c-1} q^{2 (\ell_1 + \dots + \ell_{c-k})} = {c \choose k}_{q^2}.$$

\end{proof}

The previous terms arise from spillover from previous levele $\ell '$ corresponding
to representations arising from semisimple elements at stage $\ell'$ with $-1$ an eigenvalue
of multiplicity $2 (\ell - \ell')$.
Again, summing obtains a full sum of sqares of representations of $\text{Sp}_{2(N-(m-\ell))} (\F_q)$.

Summing these error terms then gives
$$\begin{array}{c}
\displaystyle
(-1)^\ell \frac{|\text{Sp}_{2N} (\F_q)|_{q'}}{|\text{Sp}_{2(N-m+\ell )} (\F_q)|_{q'}} \frac{|\text{Sp}_{2m}(\F_q)|_{q'}}{
\text{Sp}_{2\ell}(\F_q ) \times \text{Sp}_{2(m -\ell)} (\F_q)|_{q'}} \\
\\
\displaystyle \frac{|\text{Sp}_{2(m-\ell)}(\F_q)|}{|\text{Sp}_{2(m-\ell)} (\F_q)|_{q'}} q^{N-(m-\ell)} \frac{|\text{Sp}_{2\ell} (\F_q)|}{|\text{Sp}_{2\ell} (\F_q)|_{q'}}
\end{array}
$$
which equals the $\ell$th term of \rref{Step2Expression}.
\end{proof}

\subsection{Modifications for even-dimensional orthogonal spaces}
\label{EvenCombSubsect}

In the two cases of $W$ with even dimension $\text{dim} (W) = 2m$, similar arguments for
Theorem \ref{DimsMatchTheoremEta} apply, with the following modifications:

\vspace{2mm}

\noindent {\bf Case 1: $B$ is totally split}
In this case, the orders of the parabolic quotients of
$\text{O}(W,B) = \text{O}_{2m}^+ (\F_q)$
are
$$|\text{O}_{2m}^+ (\F_q)/P_{B,k}| = {m \choose k}_q \cdot \prod_{j= m-k}^{m-1} (q^j +1)$$
for $k = 0, \dots , m$, again writing $P_{B,k}$ for the parabolic subgroup of $\text{O}(W,B)$
with Levi subgroup $\text{O}_{2(m-k)}^+ (\F_q) \times \text{GL}_k (\F_q)$.
Again, the dimension of can be directly computed by taking the dimension of \rref{EtaCorrThmDecomp}
and recursively computing. The analogue of \rref{TopDimFormula} then is
$$
\begin{array}{c}
\text{dim} ( \omega [ V\otimes W]^{\text{top}}) = \\[1ex]
\displaystyle \sum_{i=0}^m (-1)^{m-i}
q^{m-i \choose 2} {m \choose i}_q \cdot \prod_{j=i}^{m-1} (q^j+1) \cdot q^{2iN}
\end{array}$$

\vspace{5mm}

The second step of processing the dimension of the top part of the oscillator representation,
analogous to Proposition \ref{CombinEleOdd}, is
\beg{SplitCaseStep2Analogue}{\begin{array}{c}
\text{dim} ( \omega [ V\otimes W]^{\text{top}}) = \\[1ex]
\displaystyle
\sum_{\ell =0}^m (-1)^\ell  q^{\ell (\ell -1) + (m-\ell) (m-\ell -1)} {m \choose \ell}_{q^2} \frac{(q^{m-\ell} + q^\ell)}{(q^m+1)}
\prod_{j=0}^{m-\ell -1} (q^{2(N-j)} -1).
\end{array}
}

\vspace{5mm}

The significance of the coefficients in \rref{SplitCaseStep2Analogue} (similar to \rref{WhyStep2Express}) is
that for each $\ell$,
$$\begin{array}{c}
\displaystyle {m \choose \ell}_{q^2} \frac{(q^{m-\ell } + q^\ell)}{(q^m +1)} \prod_{i=0}^{m-\ell-1} (q^{2(N-j)} -1) =\\
\\
\displaystyle \frac{1}{2} \left( \frac{|\text{SO}_{2m}^+ (\F_q)|_{q'} \cdot |\text{Sp}_{2N} (\F_q)|_{q'}}{
|\text{SO}_{2\ell}^+ (\F_q)|_{q'}\cdot |\text{SO}_{2(m-\ell)}^+ (\F_q) |_{q'} \cdot |\text{Sp}_{2(N-m+\ell)} (\F_q)|_{q'}}- \right.\\
\\
\displaystyle \left. \frac{|\text{SO}_{2m}^+ (\F_q)|_{q'} \cdot |\text{Sp}_{2N} (\F_q)|_{q'}}{
|\text{SO}_{2\ell}^- (\F_q)|_{q'}\cdot |\text{SO}_{2(m-\ell)}^- (\F_q) |_{q'} \cdot |\text{Sp}_{2(N-m+\ell )} (\F_q)|_{q'}} \right) .
\end{array}
$$

\vspace{5mm}

\noindent {\bf Case 2: $B$ is not totally split.}
In this case, the order of the parabolic quotients of $\text{O}(W,B) = \text{O}_{2m}^- (\F_q)$
are
$$|\text{O}_{2m}^- (\F_q)/ P_{B,k}| = { m -1 \choose k}_{q} \cdot \prod_{j=m-k+1}^{m} (q^j +1),$$
for $k = 0, \dots , m-1$, again writing $P_{B,k}$ for the parabolic subgroup of
$\text{O}(W,B)$ with Levi subgroup $\text{O}_{2(m-k)}^- (\F_q) \times \text{GL}_k (\F_q)$.
The analogue of \rref{TopDimFormula} then is
$$\begin{array}{c}
\text{dim} ( \omega [V\otimes W]^{top}) =\\[1ex]
\displaystyle \sum_{i=1}^{m} (-1)^{m-i} q^{m-i \choose 2} {m-1 \choose i}_q 
\cdot \prod_{j=i+1}^m (q^j+1) \cdot q^{2iN}
\end{array} $$

\vspace{5mm}

Then the second step re-expresses \rref{Step2Expression} as
\beg{Step2AnalogueNonSplit}{\begin{array}{c}
\text{dim}( \omega [ V\otimes W]^{top}) =\\[1ex]
\displaystyle 
\sum_{\ell = 0}^{m-1} (-1)^\ell q^{\ell(\ell-1) + (m-\ell) (m-\ell -1)} {m \choose \ell}_{q^2}
\frac{(q^{m-\ell} - q^\ell)}{(q^m -1)} \prod_{j= 0}^{m-\ell -1} (q^{2(N-j)} -1)
\end{array}}

\vspace{5mm}

Similarly as in the non-split case, the $\ell$th factor of
\rref{Step2AnalogueNonSplit} can be interpreted by
$$
\begin{array}{c}
\displaystyle { m \choose \ell}_{q^2} \frac{(q^{m-\ell} -q^\ell)}{(q^m-1)} \prod_{i=0}^{m-\ell -1}
(q^{2(N-j)} -1) = \\
\\
\displaystyle \frac{1}{2} \left( \frac{|\text{SO}_{2m}^- (\F_q)|_{q'}  \cdot |\text{Sp}_{2N} (\F_q)|_{q'} }{|\text{SO}_{2(m-\ell)}^- (\F_q)|_{q'} \cdot |\text{SO}_{2\ell}^+ (\F_q)|_{q'} \cdot |\text{Sp}_{2(N-m+\ell)} (\F_q)|_{q'}} - \right. \\
\\
\displaystyle \left. \frac{|\text{SO}_{2m}^- (\F_q)|_{q'}  \cdot |\text{Sp}_{2N} (\F_q)|_{q'} }{|\text{SO}_{2(m-\ell)}^+ (\F_q)|_{q'} \cdot |\text{SO}_{2\ell}^- (\F_q)|_{q'} \cdot |\text{Sp}_{2(N-m+\ell)} (\F_q)|_{q'}} \right).
\end{array}
$$

\subsection{The case of the odd orthogonal stable range}\label{TopDimCombinSubSect}
The same argument as in the previous subsections also work for a choice of reductive
dual pair $(\text{Sp}(V), \text{O}(W,B))$ in the orthogonal stable range.
The same calculation as in
Proposition \ref{SympTopPartProp} also holds in this case.

\begin{proposition}\label{LemmaPartOddOrthoStabTopDim}
Consider symplectic and orthogonal spaces $V$ and $(W,B)$ whose dimensions
are in the orthogonal stable range.
The dimension of the top part of $ \omega [ V\otimes W]$ is
$$\begin{array}{c}
\text{dim} ( \omega[ V\otimes W]^{\text{top}}) = \\[1ex]
\displaystyle \sum_{i = 0}^N (-1)^{N-i} \cdot q^{N-i \choose
2}\cdot {N \choose i}_q \cdot  \prod_{j=i+1}^N (q^j+1) \cdot q^{i \cdot \text{dim} (W)}.
\end{array}$$
\end{proposition}

\noindent(Note again that nothing in the statement or proof of Proposition
\ref{LemmaPartOddOrthoStabTopDim}
uses the parity of the dimension of $W$.)

Again, we process this further:

\begin{proposition}\label{PropFinalOddOrthoStabTopDim}
Consider symplectic and orthogonal spaces $V$ and $(W,B)$
The dimension of the top part of the oscillator representation $ \omega[V \otimes W]^{\text{top}}$
is
\beg{FinalOddOrthoTopDim}{\begin{array}{c}
\text{dim} ( \omega[ V\otimes W]^{\text{top}}) = \\[1ex] 
\displaystyle \sum_{\ell=0}^N (-1)^\ell \cdot q^{(N-\ell)^2+ \ell (\ell -1)} \cdot {N \choose \ell}_{q^2}\cdot
\prod_{i=0}^{N-\ell -1} (q^{2(m-i)}-1).
\end{array}
}
\end{proposition}

\vspace{5mm}

Denote the $\ell$th term of \rref{FinalOddOrthoTopDim} by
\beg{XellNMDefn}{X_{\ell} (N, m) :=  (-1)^\ell \cdot q^{(N-\ell)^2+ \ell (\ell -1)} \cdot {N \choose \ell}_{q^2}\cdot
\prod_{i=0}^{N-\ell -1} (q^{2(m-i)}-1).}
In particular, note that
\beg{TopLevelXTerm}{X_{\ell} (\ell, m) = (-1)^\ell \cdot q^{\ell (\ell-1)}}
does not depend on $m$.
Recalling that, for any rank $r$, the order of the symplectic and odd special orthogonal
group is
$$|\text{Sp}_{2r} (\F_q)| = |\text{SO}_{2r +1} (\F_q)| = q^{r^2} \prod_{i = 1}^r (q^{2i} -1),$$
we in fact find that
\beg{InductionBetweenXS}{
\begin{array}{c}
\displaystyle  q^{(N-\ell)^2 } \cdot {N \choose \ell}_{q^2}\cdot
\prod_{i=0}^{N-\ell -1} (q^{2(m-i)}-1) =|\text{Sp}_{2(N-\ell)} (\F_q)| \cdot \\
\\
\displaystyle 
\frac{|\text{Sp}_{2N} (\F_q)|_{q'}}{
|\text{Sp}_{2(N-\ell)} (\F_q) \times \text{Sp}_{2\ell} (\F_q)|_{q'}} \cdot \frac{|\text{SO}_{2m+1} (\F_q)|_{q'}}{|\text{SO}_{2(m-N+\ell)+1} (\F_q) \times \text{SO}_{2(N-\ell)+1} (\F_q)|_{q'}}.
\end{array}}
In particular, using \rref{TopLevelXTerm}, we find that
\beg{XellsOddtoOddFinalLevelRecursion}{\begin{array}{c}
X_{\ell} (N, m) = X_{\ell} (\ell, m-N+ \ell) \cdot |\text{Sp}_{2(N-\ell)} (\F_q)| \cdot \\
\\
\displaystyle 
\frac{|\text{Sp}_{2N} (\F_q)|_{q'}}{
|\text{Sp}_{2(N-\ell)} (\F_q) \times \text{Sp}_{2\ell} (\F_q)|_{q'}} \cdot \frac{|\text{SO}_{2m+1} (\F_q)|_{q'}}{|\text{SO}_{2(m-N+\ell)+1} (\F_q) \times \text{SO}_{2(N-\ell)+1} (\F_q)|_{q'}}.
\end{array}}

\vspace{5mm}

It remains to produce the terms $X_\ell (N, m)$ from the summands on
the right hand side of \rref{DimsForZetaIsPsi}.

We will recursively compute
\beg{RHSofCombinOdd}{\sum_{\rho \in \widehat{\text{Sp}(V)}} \text{dim} (\rho ) \cdot
\text{dim}(\psi_{V}^{W,B} (\rho))}
using a series of $N$ increasingly accurate approximations.
For $\ell = 0, \dots , N$, the ``level $\ell$" approximation will be equal to
$$X_0 (N,m ) + X_1 (N,m ) + \dots + X_\ell ( N, m),$$
and will correctly count the terms
\beg{TrueTerms}{\text{dim} (\rho ) \cdot \text{dim}(\psi_{V}^{W,B} (\rho))}
for $\rho$ with Lusztig data consisting of a conjugacy class of a semisimple
element $s \in \text{SO}_{2N+1} (\F_q)$ with eigenvalue $-1$
occuring with multiplicity less than or equal to $2\ell$.

\begin{definition}
Say a representation $\rho$ of a finite group of Lie type
{\em occurs at level $\ell$} if the conjugacy class $(s)$
of a semisimple element in its Lusztig data has eigenvalue $-1$ with multiplicity $2\ell$.
\end{definition}

The level $\ell$ approximation of \rref{RHSofCombinOdd} will also generate
some error terms that must be accounted for in approximations at later levels.
At level $\ell = N$, we will have used all previous levels' errors, and correctly counted
the contribution of every $\rho \in \widehat{\text{Sp}_{2N} (\F_q)}$.

First, we describe the level $0$ approximation of \rref{RHSofCombinOdd}.
Consider irreducible representations $r^{\text{Sp} (V)} [(s), u]$ where $s$ is a conjugacy
class of a semisimple element with no $-1$ eigenvalues. We
then have
\beg{No-1EigenvsPsiOdd}{Z_{\text{SO}_{2m+1} (\F_q)} (\psi (s))^\circ = (Z_{\text{Sp}_{2N} (\F_q)} (s)^\circ)^* \times \text{Sp}_{2(m-N)} (\F_q),}
$\psi (u) = \widetilde{u} \otimes 1$ (where $1$ 
denotes the trivial representation of $\text{Sp}_{2(m-N)} (\F_q)$). Therefore,
\beg{Level0InductionFactorTerms}{\text{dim} (\psi_{V}^{W,B} (\rho)) =
\frac{|\text{Sp}_{2m} (\F_q)|_{q'}}{|\text{Sp}_{2(m-N)} (\F_q) \times \text{Sp}_{2N} (\F_q)|_{q'}}  \text{dim} (\rho).}
We define the level $0$ approximation of \rref{RHSofCombinOdd}, by
imagining that \rref{Level0InductionFactorTerms} holds for every $\rho \in \widehat{\text{Sp}_{2N} (\F_q)}$,
giving
$$\sum_{\rho \in \widehat{\text{Sp}_{2N} (\F_q)}} \frac{|\text{Sp}_{2m} (\F_q)|_{q'}}{|\text{Sp}_{2(m-N)} (\F_q) \times \text{Sp}_{2N} (\F_q)|_{q'}}  \text{dim} (\rho)^2.$$ 
We can see that this is
$$\frac{|\text{Sp}_{2m} (\F_q)|_{q'}}{|\text{Sp}_{2(m-N)} (\F_q) \times \text{Sp}_{2N} (\F_q)|_{q'}} 
 |\text{Sp}_{2N} (\F_q)| = X_N (0, N).
$$
The error of the level $0$ approximation consits of two kinds of contributions
for $\rho$ occuring at level $1 \leq \ell \leq N$:
the ``true terms" \rref{TrueTerms},
and the negative of the ``faked terms added at level $0$," which are precisely
\beg{Level0FakedTerms}{-\frac{|\text{Sp}_{2m} (\F_q)|_{q'}}{|\text{Sp}_{2(m-N)} (\F_q) \times \text{Sp}_{2N} (\F_q)|_{q'}} \text{dim} (\rho)^2.}

\vspace{5mm}

Now let us consider the level $\ell$ approximation for $1 \leq \ell \leq N$.
For a representation $r^{\text{Sp} (V)}[s, u, \pm 1]$ occuring at level $\ell$, we have
$$Z_{\text{SO}_{2N+1} (\F_q)} (s)^\circ = H \times \text{SO}_{2\ell}^\pm (\F_q),$$
where we may consider $H$ as the identity component of the centralizer of a semisimple element $s'$, which is conjugate to 
the diagonalization of $s$ restricted away from the $2\ell$ coordinates
with eigenvalues $-1$, in $\text{SO}_{2(N-\ell)+1} (\F_q)$:
\beg{CentralizerLevellIntermediate}{H = Z_{\text{SO}_{2(N-\ell)+1} (\F_q)} (s')^\circ.}
The identity components of the centralizers of semisimple elements of $\text{SO}_{2(N-\ell) +1} (\F_q)$ which appear
as \rref{CentralizerLevellIntermediate} are precisely those with no factors of type $D$ or ${}^2D$
(since $s'$ by definition has no $-1$ eigenvalues).
Write a unipotent representation $u$ of $H \times \text{SO}_{2\ell}^\pm (\F_q)$ as
$$u = u_H \otimes u_{\text{SO}_{2\ell}^\pm (\F_q)},$$
for $u_H \in \widehat{H}_u$,
$u_{\text{SO}_{2\ell} ^\pm  (\F_q)} \in \widehat{\text{SO}_{2\ell}^\pm (\F_q) }_u$.
Then the sum of the true terms contributed by $r^{\text{Sp} (V)}[(s),u, +1]$ and 
$r^{\text{Sp} (V)}[(s),u, -1]$ is
the product of the ``induction factor"
\beg{IndFactorGeneralLevelLIntermediateH}{\frac{|\text{Sp}_{2N} (\F_q)|_{q'} \cdot |\text{Sp}_{2m} (\F_q)|_{q'}}{|H
\times \text{SO}_{2\ell}^\pm (\F_q)|_{q'} \cdot |H \times \text{SO}_{2(m-n+\ell) +1} (\F_q)|_{q'}}\text{dim} (u_H)^2}
with
$$\begin{array}{c}
\displaystyle \frac{\text{dim} (u_{\text{SO}_{2\ell}^\pm (\F_q)})}{2}
\cdot (\text{dim} (\psi^{+1} (u_{\text{SO}_{2\ell}^\pm (\F_q)}) + \psi^{-1} (u_{\text{SO}_{2\ell}^\pm (\F_q)})).
\end{array}$$
Now \rref{IndFactorGeneralLevelLIntermediateH} can be re-written as
\beg{LevelL}{
\begin{array}{c}
\displaystyle
\frac{|\text{Sp}_{2N} (\F_q)|_{q'} |\text{Sp}_{2m} (\F_q)|_{q'}}{|\text{SO}_{2(N-\ell) +1}(\F_q)
\times \text{SO}_{2\ell}^\pm (\F_q)|_{q'} }\cdot \vspace{1mm}\\
\displaystyle \frac{ |\text{Sp}_{2m} (\F_q)|_{q'}}{ |\text{SO}_{2(N-\ell) +1}(\F_q) \times \text{SO}_{2(m-N+\ell) +1} (\F_q)|_{q'}}\cdot \vspace{1mm} \\
\displaystyle \text{dim} (r^{\text{SO}_{2(N-\ell)+1} (\F_q)}[(s'), u_H])^2,
\end{array}}
where $r^{\text{SO}_{2(N-\ell)+1} (\F_q)} [(s'), u_H]$ denotes the irreducible representation of $\text{SO}_{2(N-\ell)+1} (\F_q)$
associated to the $\text{SO}_{2(N-\ell)+1} (\F_q)$-classification data of $[(s'), u_H]$.
We introduce ``faked terms occuring at level $\ell$" which consist of a product of
\rref{LevelL} with $X_{\ell} (\ell ,N)$.

\vspace{3mm}

Hence, by induction on $N$, this reduces
\rref{DimsForZetaIsPsi} to checking the ``highest level" of singularity, i.e.
find terms matching the $N$th term. The ``true" new representations obtained at level $N$
arise from Lusztig data
$$\textstyle [\sigma_m^\pm, {\lambda_1 < \dots < \lambda_a \choose \mu_1< \dots < \mu_b}, \pm 1],$$
recalling Definition \rref{SigmanDefn},
where ${\lambda_1 < \dots < \lambda_a \choose \mu_1< \dots < \mu_b}$
denotes a symbol
specifying a unipotent representation of $\text{SO}_{2N}^\pm (\F_q)$.

\begin{proposition}
The sum of the ``true" level $N$ terms
$$ \begin{array}{c}
\displaystyle \sum_{u \in \widehat{\text{SO}_{2N}^+ (\F_q)}_u} \text{dim} (r^{\text{Sp}_{2N} (\F_q)}[\sigma_N^+, u, \pm1]) \cdot \text{dim} (\psi_{V}^{W,B} (r^{\text{Sp}_{2N} (\F_q)}[\sigma_N^+, u, \pm1]))+
\\
\displaystyle \sum_{u \in \widehat{\text{SO}_{2N}^- (\F_q)}_u} \text{dim} (r^{\text{Sp}_{2N} (\F_q)}[\sigma_N^-, u, \pm1]) \cdot \text{dim} (\psi_{V}^{W,B} (r^{\text{Sp}_{2N} (\F_q)}[\sigma_N^-, u, \pm1])) 
\end{array}$$
(where we sum over both central signs where left ambiguous)
and every level $\ell$ error contribution for $1 \leq \ell \leq N-1$
to the $N$th level
$$(-1)^{\ell +1} \cdot \left(\sum_{u \in \widehat{\text{SO}_{2(N-\ell)}^+ (\F_q)}_u } \text{dim} (u)^2+ \sum_{u \in \widehat{\text{SO}_{2(N-\ell)}^- (\F_q)}_u } \text{dim} (u)^2\right) \cdot X_{N} (\ell , m)$$
is equal to
$$X_{N}(N,m) = q^{N (N-1)}.$$
\end{proposition}

\begin{proof}
First, suppose ${\lambda_1< \dots < \lambda_a \choose \mu_1< \dots < \mu_b}$ is
a non-degenerate symbol of $\text{SO}_{2N}^\pm (\F_q)$. Let us write
$$x:= N-m + \frac{a+b}{2}$$
Then the sum of dimensions
$$
\begin{array}{c}
\text{dim}_{\text{SO}_{2N+1} (\F_q)} ({\lambda_1< \dots < \lambda_a<x \choose \mu_1< \dots < \mu_b}) +
\text{dim}_{\text{SO}_{2N+1} (\F_q)} ({\lambda_1< \dots < \lambda_a \choose \mu_1< \dots < \mu_b < x})
\end{array}$$
is equal to the product of
\beg{}{\textstyle (q^N \pm 1)\cdot \text{dim}_{\text{SO}_{2N}^\pm (\F_q)} ( {\lambda_1< \dots < \lambda_a \choose \mu_1 < \dots < \mu_b})}
with the factor
\beg{FalseInductionFactor}{
\frac{\displaystyle \prod_{i=1}^a (q^{x} -q^{\lambda_i} ) \prod_{j=1}^b (q^{x}+ q^{\mu_j}) + \prod_{i=1}^a (q^{x}+ q^{\lambda_i}) \prod_{j=1}^b(q^{x} -q^{\mu_j})}{\displaystyle\prod_{i=1}^{(a+b)/2} (q^{x} -q^{i})(q^{x} +q^i)}.
}
The top $q$-degrees of the numerator and denominator of \rref{FalseInductionFactor}
clearly match, and are equal to $x(a+b)$, suggesting a cancellation with the corresponding ``level 0"
error term.
Our goal is to re-express the numerator of \rref{FalseInductionFactor} in terms of the previous
levels' error terms. To do this, we proceed inductively, replacing each error term's
$X_N(\ell, m)$ factor with the induction hypothesis for $X_{\ell} (\ell , m)$, multiplied
by \rref{InductionBetweenXS}.
This will give
$$\frac{(-1)^N}{2} \cdot \left( \frac{\sum_{\rho \in \widehat{\text{SO}_{2N}^+ (\F_q)} } \text{dim}(\rho)^2 }{|\text{SO}_{2N}^+ (\F_q)|_{q'}} + 
\frac{\sum_{\rho \in \widehat{\text{SO}_{2N}^- (\F_q)} } \text{dim}(\rho)^2}{|\text{SO}_{2N}^- (\F_q)|_{q'}} \right),$$
which is $(-1)^N q^{N (N-1)}$ since the $q$ part of the order of $|\text{SO}_{2N}^\pm (\F_q)|$ is $q^{N(N-1)}$.

\end{proof}

\subsection{Modifications for even orthogonal groups}\label{EvenCombinSubSect}

Now consider orthogonal spaces $W$ of even dimension $\text{dim} (W) =  2m$.
First, note that there is no distinction in the dimension of the top part depending on whether
the symmetric bilinear form on $W$ is completely split or not. Our replacement for the calculation of the
dimension of the top part is

\begin{proposition}
Suppose $\text{dim} (W) = 2m$, $\text{dim } (V)= 2N$. Then
\beg{EvenDimPropTopDim}{\begin{array}{c}
 \omega [ V\otimes W]^{\text{top}} =\\[1ex]
\displaystyle \sum_{\ell= 0}^N 
q^{\ell^2 + (N-\ell) (N-\ell -1)} {N \choose \ell}_{q^2} \cdot \prod_{i=1}^{N-\ell} (q^{2(m-N + \ell + i)} -1)
\end{array}}
\end{proposition}
\noindent Write $Y_{\ell} (N,m)$ for the $\ell$th term of \rref{EvenDimPropTopDim}, replacing \rref{FinalOddOrthoTopDim}.

Let us suppose that the symmetric bilinear form on $W$ is completely split, i.e.
$\text{O}(W,B)= \text{O}_{2m}^+ (\F_q)$. (Again the non-split even case follows similarly.)
Consider a semisimple element $s$ of $\text{SO}_{2N+1} (\F_q)$ with $1$
as an eigenvalue of total multiplicity $2\ell +1$. Then the identity component of its centralizer is of the form
\beg{CentralizerOfSSEltForModifEvenSect}{Z_{\text{SO}_{2N+1} (\F_q)} (s)^\circ = H \times \text{SO}_{2\ell+1} (\F_q),}
for $H$ now denoting 
the identity component of a centralizer of a semisimple element with no $1$ eigenvalues
in an even special orthogonal group
(of either parity) $\text{SO}_{2(N-\ell)}^\pm (\F_q)$. Let us write $H \subseteq \text{SO}_{2(N-\ell}^\epsilon (\F_q)$.
Then
\beg{ModifEvenSSectPsiCentralizerSplitCase}{Z_{\text{SO}_{2m}^+ (\F_q)} (\psi (s))^\circ = H^* \times \text{SO}_{2(m-N +\ell)}^\epsilon (\F_q).}

Hence, inductively, the level $\ell$ approximation in this case has terms
equal to $Y_{\ell} (\ell, m)$, multiplied by ``inductive factor" equal to 
half of the sum of 
\beg{InductiveSplitOrthoLevelEllFactor}{\begin{array}{c}
\displaystyle |\text{SO}_{2(N-\ell)}^\epsilon (\F_q)| \frac{|\text{SO}_{2N+1} (\F_q)|_{q'}}{|\text{SO}_{2(N-\ell)}^\epsilon (\F_q) \times \text{SO}_{2\ell +1} (\F_q)|_{q'}}\\
\\
\displaystyle  \frac{|\text{SO}_{2m}^+ (\F_q)|_{q'}}{|\text{SO}_{2(N-\ell)}^\epsilon (\F_q) \times \text{SO}_{2(m-N+\ell)}^\epsilon (\F_q)|_{q'}}
\end{array}
}
over the two choices of $\epsilon = \pm $. Now \rref{InductiveSplitOrthoLevelEllFactor}
can be simplified as
\beg{SimplifiedInductiveFactorSplitOrtho}{\begin{array}{c}
\displaystyle q^{(N-\ell) (N-\ell -1) } \cdot {N \choose \ell}_{q^2} \cdot (q^m -1) \cdot \prod_{i =1}^{N-\ell -1} 
(q^{2(m-N+\ell +i)} -1) \cdot \\
\\
 (q^{N-\ell} + \epsilon 1) \cdot (q^{m-N+\ell} + \epsilon 1),
\end{array}}
and we further have
$$\frac{1}{2}((q^{N-\ell} + 1) \cdot (q^{m-N+\ell} +1) + (q^{N-\ell} -1) \cdot (q^{m-N+\ell} -1))
= q^m +1.$$
Therefore, the average of \rref{SimplifiedInductiveFactorSplitOrtho}
over the two choices of parity $\epsilon = \pm $ is
\beg{FinalSplitEvenOrtho}{q^{(N-\ell) (N-\ell -1)} \cdot {N \choose \ell}_{q^2} \prod_{i=1}^{N-\ell} (q^{2(m-N+\ell + i)} -1).}
Hence, considering \rref{EvenDimPropTopDim}, it remains to find
\beg{YTerms}{Y_{\ell} (\ell, m) = q^{\ell^2}.}
Finding these terms proceeds exactly similarly to in the case of odd-dimensional $W$, since
it is the $q$-part of the order of $\text{SO}_{2\ell+1}(\F_q)$.

\vspace{5mm}

In the case when the symmetric bilinear form on $W$ is not completely split, i.e.
$\text{O} (W,B) = \text{O}_{2m}^- (\F_q)$, if we have \rref{CentralizerOfSSEltForModifEvenSect},
then instead of \rref{ModifEvenSSectPsiCentralizerSplitCase}, we have
$$Z_{\text{SO}_{2m}^- (\F_q)} (\psi (s))^\circ = H^* \times \text{SO}_{2(m-N+\ell)}^{-\epsilon} (\F_q)$$
and therefore the inductive factor \rref{InductiveSplitOrthoLevelEllFactor}
is replaced by
\beg{InductiveNonSplitOrthoLevelEllFactor}{\begin{array}{c}
\displaystyle |\text{SO}_{2(N-\ell)}^\epsilon (\F_q)| \frac{|\text{SO}_{2N+1} (\F_q)|_{q'}}{|\text{SO}_{2(N-\ell)}^\epsilon (\F_q) \times \text{SO}_{2\ell +1} (\F_q)|_{q'}}\\
\\
\displaystyle  \frac{|\text{SO}_{2m}^- (\F_q)|_{q'}}{|\text{SO}_{2(N-\ell)}^\epsilon (\F_q) \times \text{SO}_{2(m-N+\ell)}^{-\epsilon} (\F_q)|_{q'}},
\end{array}
}
which is simplified as 
$$\begin{array}{c}
\displaystyle q^{(N-\ell) (N-\ell -1) } \cdot {N \choose \ell}_{q^2} \cdot (q^m +1) \cdot \prod_{i =1}^{N-\ell -1} 
(q^{2(m-N+\ell +i)} -1) \cdot \\
\\
 (q^{N-\ell} + \epsilon 1) \cdot (q^{m-N+\ell} - \epsilon 1).
\end{array}$$
Now we have
$$\frac{1}{2}((q^{N-\ell} + 1) \cdot (q^{m-N+\ell} -1) + (q^{N-\ell} -1) \cdot (q^{m-N+\ell} +1))
= q^m -1,$$
again simplifying the average of terms for different parities $\epsilon = \pm$ into \rref{FinalSplitEvenOrtho},
meaning that it remains to find the same terms \rref{YTerms}.

\section{An inductive argument}\label{InductiveProof}

In this section, we conclude the statement of Theorem \ref{EtaExplicitIntro}.
First, we note that the toral characters of the eta and zeta correspondence are determined
inductively, by examining the restriction of the oscillator representations to finite general linear groups.
This confirms that the semisimple and cetnral sign data of $\eta^V_{W,B} (\rho)$ (resp.
$\zeta^{W,B}_V (\rho)$) matches that of $\phi^V_{W,B} (\rho)$ (resp. $\psi^{W,B}_V (\rho)$).
This is treated in Subsection \ref{SemisimpleSignPartsMatch}.

It then remains in all cases to confirm the unipotent part of $\eta^V_{W,B} (\rho)$
(resp. $\zeta^{W,B}_V (\rho)$) matches that of $\phi^V_{W,B} (\rho)$ (resp. $\psi^{W,B}_V (\rho)$).
First, we prove Proposition \ref{NnRank}, and conclude that for $N>>n$, we have
\beg{DimEtaIsDimPhi}{\text{dim} (\eta^V_{W,B} (\rho)) = \text{dim} (\phi^V_{W,B} (\rho))}
(and similarly, for $n>>N$, we have 
\beg{DimZetaIsDimPsi}{\text{dim} (\zeta^V_{W,B} (\rho)) = \text{dim} (\psi^V_{W,B} (\rho))).}
We may view these dimensions as polynomials of $q^N$ (resp. $q^n$).
The results of \cite{TotalHoweI} can be used to see that in either stable range,
the idempotent in the endomorphism algebra picking out any
summand of the eta (resp. zeta) correspondence does not depend on $N$ (resp. $n$).
Therefore, we can apply the description from \cite{TotalHoweI} to see that \rref{DimEtaIsDimPhi}
and \rref{DimZetaIsDimPsi} both hold for any choice of $N$, $n$ in the symplectic and orthogonal
stable ranges. Therefore, since each unipotent representation corresponding to a different
symbol has a different dimension,
we find that our claimed construction is the only possible choice. Hence, we conclude Theorem
\ref{EtaExplicitIntro}.

For the remainder of this section, we restrict attention to the case of the eta correspondence and
$\phi^V_{W,B}$, since the case of the zeta correspondence and $\psi^{W,B}_V$ can be
done completely similarly.

\subsection{Determining the semisimple and sign data}\label{SemisimpleSignPartsMatch}
The purpose of this subsection is to prove that the semisimple part (and sign data)
of the $\text{Sp}(V)$-classification data of the representation obtained by applying an eta
correspondence $\eta^V_{W,B} (\rho)$ matches the semisimple part (and sign data)
of our constructed representation $\phi^V_{W,B} (\rho)$ (and the similar statement
for $\zeta^{W,B}_V$ and $\psi^{W,B}_V$).

Broadly, this can be concluded since, considering $\text{GL}_N (\F_q) \subseteq  \text{Sp}(V)$,
the restriction of the oscillator representation is
$$\text{Res}_{\text{GL}_N(\F_q)} ( \omega[V]) \cong \epsilon (det) \otimes \C \F_q^N .$$
Now we also have the restriction
$$\text{Res}_{\text{GL}(V)} ( \omega [V \otimes W])  \cong (\text{Res}_{\text{GL}(V)} ( \omega [V]))^{\otimes W}$$
where $\otimes W$ denotes a degree $\text{dim} (W)$
tensor product of oscillator representations $ \omega [V]$.
Since characters are matched exactly in the premutation representation factors,
for example in the case of odd-dimensional $W$, we know the underlying
toral character and the sign data.
We now restrict attention to
the case of comparing $\eta^V_{W,B}$ and $\phi^V_{W,B}$, for $(\text{Sp}(V), \text{O}(W,B))$
in the symplectic stable range. The case of comparing $\zeta^{W,B}_V$ and $\psi^{W,B}_V$
for the orthogonal stable range is similar.

\vspace{3mm}

\begin{proposition}\label{SameCharacters}
Suppose $(\text{Sp}(V), \text{O}(W,B))$ is in the symplectic stable range.
If $\text{dim} (W) = 2m+1$ is odd,
for $\rho$ an irreducible representation of $\text{SO}_{2m+1}(\F_q)$ arising from the
conjugacy class of a semisimple element $s \in \text{SO}_{2m+1} (\F_q)$
and a unipotent representation $u$ of the dual of the identity component of its centralizer, then in the
classification data of
$\eta^V_{W,B} ((\pm 1) \otimes \rho)$, its semisimple part is
$$(\phi^\pm (s)) = (s \oplus \sigma^\pm_{N-m}).$$
If $\text{dim} (W) = 2m$ is even, for $\rho$ an irreducible representation of $\text{O}_{2m}^\pm (\F_q)$ arising from 
an $\text{SO}_{2m}^\pm (\F_q)$-representation corresponding to the
conjugacy class of a semisimple element $s \in \text{SO}_{2m}^\pm (\F_q)$
and a unipotent representation $u$ of the dual of the identity component of
its centralizer, then in the Jordan
decomposition of
$\eta^V_{W,B} (\rho)$, its semisimple part is
$$(\phi (s)) = (s \oplus I_{2(N-m)+1}).$$ 
\end{proposition}

\begin{proof}Suppose $\text{dim} (W) = 2m+1$.
Let us begin by considering
\beg{mSO2+}{\underbrace{\text{SO}_2^\pm (\F_q) \times \dots \times \text{SO}_2^\pm (\F_q)}_{m}}
as a torus of $\text{SO}(W,B)$. Fix a character
$$\chi_{a_1} \otimes \dots \otimes \chi_{a_m},$$
corresponding to $a_1, \dots , a_m \in \mu_{q\mp1} \cong \text{SO}_2^\pm (\F_q)$. Consider the maximal
parabolic subgroup with Levi \rref{mSO2+} (i.e. the Borel subgroup) $B(W,B)\subseteq \text{SO} (W,B)$.
Then, for an irreducible representation $\rho$ with this character, i.e.
$$\rho \subseteq \text{Ind}^{\text{O}(W,B)} (\chi_{a_1} \otimes \dots \otimes \chi_{a_m}),$$
we need to prove that
$\eta_{W,B} (\rho)$ corresponds to a toral character 
\beg{ClaimedSpVsCharacterCorr}{\chi_{a_1} \otimes \dots \otimes \chi_{a_m} \otimes (\epsilon)^{\otimes N-m}}
in $\underbrace{\text{SO}_2^\pm (\F_q) \times \dots \times \text{SO}_2^\pm (\F_q)}_{N}\subseteq \text{Sp}_{2N} (\F_q)$
(considering $\epsilon$ as the quadratic character of $\mu_{q\mp 1} = \text{SO}_2^\pm (\F_q)$.

\vspace{3mm}

Consider the inclusion of the product of this torus with $\text{Sp} (V)$
\beg{TorusAndSpVS}{\begin{array}{c}
\underbrace{\text{SO}_2^\pm (\F_q) \times \dots \times \text{SO}_2^\pm (\F_q)}_m \times \text{Sp} (V)
\subseteq \text{SO} (W,B) \times \text{Sp} (V) \\[1ex]
\subseteq \text{Sp} (V \otimes W).
\end{array}}
Pick the $i$th factor $\text{SO}_2^\pm (\F_q)$ in \rref{TorusAndSpVS}, taking the inclusion
\beg{deg2Inclusion}{\text{SO}_2^\pm (\F_q) \times \text{Sp}(V) \subseteq \text{Sp} (V \otimes W)}
Restricting $ \omega [ V \otimes W]$ along \rref{deg2Inclusion} gives a restriction
\beg{Degree2HoweDuality}{\text{Res}_{\text{SO}_2^\pm (\F_q) \times \text{Sp}(V)} ( \omega [ V \otimes \F_q^2]) \otimes \C^{q^{(2m-1)N}}}
considering $\F_q^2$ with the split and non-split symmetric bilinear form, respectively,
(and taking the trivial action on $\C^{q^{(2m-1)N}}$.
Recalling the results of \cite{TotalHoweI}, in each factor \rref{Degree2HoweDuality}, it decomposes as
a $\text{SO}_2^\pm (\F_q) \times \text{Sp}(V)$-representation pairing every $\chi_{a_i}$-type
$\text{SO}_2^\pm (\F_q)$-representation with a representation $\text{Sp}(V)$ in the induction
$$\text{Ind}_{\text{SO}_2^\pm (\F_q)} (\chi_{a_i}),$$
considering $\text{SO}_2^\pm (\F_q)$ as a factor of a torus in $\text{Sp}(V)$
Since this holds for every $i$, it also holds in the restriction of $ \omega [ V \otimes W]$
along \rref{TorusAndSpVS}: in
$$\text{Res}_{\text{SO}_2^\pm (\F_q) \times \dots \times \text{SO}_2^\pm (\F_q) \times \text{Sp} (V)} ( \omega [ V \otimes W]),$$
the character $\chi_{a_1} \otimes \dots \otimes \chi_{a_m}$ as a representation
of $\text{SO}_2^\pm (\F_q) \times \dots \times \text{SO}_2^\pm (\F_q)$
is paired with a representation of $\text{Sp}(V)$ in that character's induction,
viewing the copies of $\text{SO}_2^\pm (\F_q)$'s as blocks in a torus of $\text{Sp} (V)$.

The remaining factors of $\epsilon$ in \rref{ClaimedSpVsCharacterCorr} corresponding to the remaining $N-m$
factors in a torus of $\text{Sp}(V)$ arise since the restriction of $\text{Sp}(V)$ to a representation of 
$$\text{GL}_{N-m} (\F_q) \subseteq \text{GL} (\Lambda) \subseteq \text{Sp}(V)$$
is $\epsilon (det)$ tensored with a permutation representation.

\vspace{3mm}

A similar argument applies to both even-dimensional cases.
 
\end{proof}

\subsection{The proof of Propostion \ref{NnRank}}

The purpose of thie subsection is 
to prove Propostion \ref{NnRank} by induction.
Again, we restrict attention to the case of $N>>n$, since the
case of $n>>N$ is completely similar.

\vspace{3mm}

First, we begin by observing the following

\begin{lemma}\label{All2m+1NRank}
Fix $n$, and consider $N>>n$. Every irreducible representation of $\text{Sp}_{2N} (\F_q)$ with $N$-rank $n$ is
constructed
by applying
$\phi_{W,B}^V$ to an irreducible representation of $\text{O}(W,B)$ for $n$-dimensional orthogonal space $(W,B)$.
\end{lemma}

\begin{proof}
First suppose $\text{dim} (W) = n =2m+1$.
Writing out the definition of $\phi_{W,B}$, we find that the statement is equivalent to the claim that
every irreducible representation of $\text{Sp}_{2N} (\F_q)$ of $N$-rank $2m+1$ arises from a
conjugacy class $(s)$ of a semisimple element of $\text{SO}_{2N+1} (\F_q)$ with 
the identity component of its centralizer expressible as
\beg{PhiWBCentra}{\prod_{i=1}^r U_{j_i}^+ (\F_q) \times \prod_{i=1}^t U_{k_i}^- (\F_q)\times
\text{SO}_{2\ell+1} (\F_q) \times \text{SO}_{2(N-m+p)}^\pm (\F_q)}
and a unipotent representation $u$ of this group,
whose $\text{SO}_{2(N-m+p)}^\pm (\F_q)$-representation tensor
factor $u_{\text{SO}_{2(N-m+p)}^\pm}$ corresponds to a symbol of the form
${\alpha_1< \dots < \alpha_a \choose \beta_1< \dots < \beta_b}$
such that either
$$\alpha_a = N-m+p + \frac{a+b-1}{2} \text{ or }\beta_b =N-m+p + \frac{a+b-1}{2}.$$

First note the prime to $q$ part of the group orders
$$|\text{SO}_{2N+1} (\F_q)|_{q'} = \prod_{i=1}^N (q^{2i}-1), \hspace{3mm}
|\text{SO}_{2\ell+1} (\F_q)|_{q'} = \prod_{i=1}^\ell (q^{2i} -1),$$
for the groups of type $B$
$$|\text{SO}_{2 (N-m+p)}^\pm (\F_q)|_{q'} = (q^{N-m+p} \mp 1)  \prod_{i=1}^{N-m+p-1} (q^{2i} -1),$$
for the group of type $D$,
and
$$|U_{k_i}^\pm (\F_q)|_{q'} = \prod_{u=1}^{k_i} (q^u -(\pm 1)^u) \text{ for } i = 1, \dots , t.$$

Therefore, the total top degree of $q$ in the quotient of prime
to $q$ parts of group orders \rref{indexSp2N} is
$$\sum_{i=1}^N 2i - (\sum_{i=1}^\ell 2i + (N-m+p) + \sum_{i=1}^{N-m+p-1} 2i + 
\sum_{i=1}^r \sum_{u=1}^{j_i} u + \sum_{i=1}^t \sum_{u=1}^{k_i} u),$$
which can be simplified as
\beg{TopQDegreeInProofOfOnly}{\begin{array}{c}
N(N+1) - (\ell (\ell+1) + (N-m+p)^2 +\\[1ex]
\displaystyle  \sum_{i=1}^r\frac{j_i (j_i+1)}{2} 
+\sum_{i=1}^t \frac{k_i (k_i+1)}{2}).
\end{array}}
The terms not involving $N$ (arising from $\text{SO}_{2\ell+1}(\F_q)$ and the unitary groups)
do not affect the $N$-rank of the final $\text{Sp}_{2N} (\F_q)$-representation, since
$$\ell +\sum_{i=1}^r j_i + \sum_{i=1}^t k_i \leq m < \frac{N}{2}.$$
The remaining terms of \rref{TopQDegreeInProofOfOnly} are
$$N \cdot (1-2(m-p)) + (m-p)^2.$$
Therefore, no smaller factor of type $D$ can occur than those allowed by \rref{PhiWBCentra}.

The condition on the symbol arises since otherwise the factor \rref{DTypeFactorSymbol} contributes
additional copies of $N$, 
unless it is cancelled by the denominator of \rref{SymbolDimTwoRowsBigFrac}, which can only occur
if the rank $N-m+p+ (a+b-1)/2$ occurs as an entry in the symbol itself.

\vspace{3mm}

A similar argument applies to even cases of $n = \text{dim} (W)$.

\end{proof}

The case of Proposition \ref{NnRank} for $N>>n$ then follows by induction.

\begin{proof}[Proof of Proposition \ref{NnRank}, part \rref{N>>nPartOfRankProp}]
First we consider the case of $W$ with odd dimensions,
and proceed by induction.
Suppose for every $m' < m$, we know that the disjoint union of the images
of the two eta correspondences $\eta^V_{W,B}$ such that $\text{dim} (W) = 2m'+1$
is exactly the set of all irreducible representations of $\text{Sp}_{2N} (\F_q)$ with $N$-rank
$2m'+1$, for $N>>m$.

Suppose $(W,B)$ forms an orthogonal space of dimension $2m+1$.
By the definition of $\eta^V_{W,B}$, the sum
$$\bigoplus_{\rho \in \widehat{\text{O} (W,B)}} \rho \otimes \eta_{W,B}^V (\rho)$$
is the top summand of $ \omega [V \otimes W]$. In particular, its dimension
less than or equal to
$$\text{dim} ( \omega) = q^{(2m+1) N},$$
so
all $\text{Sp}_{2N} (\F_q)$-representations of higher $N$-rank cannot occur
in the image of $\eta^V_{W,B}$.
Additionally, the images of the different $\eta$-correspondences are all disjoint.
Therefore, by the induction hypothesis, no irreducible representations of lesser odd $N$-rank 
may occur in the image of $\eta_{W,B}$.

\vspace{3mm}

To conclude Theorem \ref{EtaExplicitIntro}, note that the pairing $\phi_{W,B}$
obtains the maximal possible dimension
$$\text{dim} (\bigoplus_{\rho \in \text{O}(W,B)} \rho \otimes \phi_{W,B} (\rho)).$$
If a representation of $\text{O}(W,B)$ were paired
by $\eta_{W,B}$ with a $\text{Sp}_{2N} (\F_q)$-representation
of lesser $N$-rank, 
it would waste dimensions in
$$\text{dim} (\bigoplus_{\rho \in \widehat{\text{O} (W,B)}} \rho \otimes \eta_{W,B} (\rho)),$$
which would be impossible to get back, by Theorem \ref{DimsMatchTheoremEta}, since no other
representations of $N$-rank $2m+1$ exist by Proposition \ref{All2m+1NRank}.
\end{proof}

\subsection{Concluding Theorem \ref{EtaExplicitIntro}}\label{FinalArgumentSubSect}
In this subsection, we first conclude
that for every $\rho \in \widehat{\text{O}(W,B)}$, 
\beg{DimEtaIsDimPhi}{\text{dim} (\eta^V_{W,B} (\rho)) = \text{dim} (\phi^V_{W,B} (\rho)).}
for $V$ of dimension $2N$ and $W$ of dimension $n$, with $N>>n$.
In our construction, for a fixed choice of $(W,B)$ and $\rho$, for every $N\geq n$,
the dimension of our constructed
representation $\phi^V_{W,B} (\rho)$ for $\text{dim} (V) = 2N$ can be expressed as a polynomial
of $q^N$ (see \rref{PolynomialPhiDim} below).
On the other hand, we recall the results of \cite{TotalHoweI}, which allow us to consider
the eta correspondence on the level of idempotents. By the stable description of
the endomorphism algebra of an oscillator representation given in \cite{TotalHoweI},
we also know the dimensions of $\eta^V_{W,B} (\rho)$ for a fixed $\rho$ and $(W,B)$
must be polynomial in $q^N$.
Therefore, \ref{DimEtaIsDimPhi} must in fact hold for every $N \geq n$.
Combining this with the results of the previous subsection, we conclude that
$$\eta^V_{W,B} (\rho) = \phi^V_{W,B} (\rho),$$
since all symbols have different dimensions.

\vspace{3mm}

First, combining Proposition \ref{SameCharacters}, Proposition \ref{NnRank}, and Theorem \ref{DimsMatchTheoremEta}
allows us to conclude \rref{DimEtaIsDimPhi} for $N>> n$:
Our construction $\phi^V_{W,B}$ satisfies the condition that, for representations
$\rho, \pi \in \widehat{\text{O}(W,B)}$ such that $\text{dim} (\rho) < \text{dim} (\pi)$, we have
$$\text{dim} (\phi^V_{W,B} (\rho))< \text{dim} (\phi^V_{W,B} (\pi)).$$
Therefore, $\phi^V_{W,B}$ is an injective correspondence
from 
which maximizes the dimension sum
$$
\sum_{\rho \in \widehat{\text{O}(W,B)}} \text{dim} (\rho) \cdot \text{dim} (\phi_{W,B} (\rho)),
$$
which we know numerically matches with
$$\sum_{\rho \in \widehat{\text{O}(W,B)}} \text{dim} (\rho) \cdot \text{dim} (\eta_{W,B} (\rho))$$
by Theorem \ref{DimsMatchTheoremEta}.
Therefore, for $N>>n$, we must have that the dimensions of $\eta_{W,B}^V (\rho)$
match the dimensions of $\phi_{W,B}^V (\rho)$. It remains to prove that this holds for
every $N\geq n$, from which we can conclude that the unipotent parts of their classification
data agree in general. 
We do this now, concluding Theorem \ref{EtaExplicitIntro}, art \rref{SympStabSide}.
The proof of Part \rref{OrthoStabSide} is similar, using the analogue of Proposition \ref{SameCharacters}
for the zeta correspondence, and the orthogonal stable cases of
Proposition \ref{NnRank}, and Theorem \ref{DimsMatchTheoremEta}.

\begin{proof}[Proof of Theorem \ref{EtaExplicitIntro}, part \rref{SympStabSide}]
We restrict attention to the case of $W$ odd dimensional. The even dimensional case
proceeds similarly. 
Fix an orthogonal space $(W,B)$ of dimesion $n = 2m+1$,
and fix an irreducible representation $\rho$ of $\text{O}(W,B)$. Considering
$\text{O}(W,B) = \Z/2 \times \text{SO}_{2m+1} (\F_q)$, write $\rho$ as a tensor product
$$\rho =r^{\text{O} (W,B)} [(s),u]^\alpha$$
for $\alpha$ denoting a sign specifying a $\Z/2$-action, and $[(s), u]$
denoting the $\text{SO}_{2m+1} (\F_q)$-classification data corresponding to the restriction of
$\rho$ to $\text{SO}_{2m+1} (\F_q)$. 
Let us consider the symbol ${\lambda_1< \dots < \lambda_a \choose \mu_1< \dots < \mu_b}$
associated to the factor of $u$ corresponding to the $-1$ eigenvalues of $s$, as in the construction
of $\phi^V_{W,B} (\rho)$.
Recall the notation $N_\rho' = N-m + \frac{a+b-1}{2}$.
For every $V$ of dimension $2N$ with $N \geq 2m+1$, the dimension of $\phi^V_{W,B} (\rho)$
is then equal to 
\beg{PolynomialPhiDim}{
\frac{\displaystyle \text{dim} (\rho) \cdot \prod_{i=N_\rho' +1}^N (q^{2i} -1) \cdot
\prod_{i=1}^a (q^{N_\rho'}+ \alpha \cdot  q^{\lambda_i}) \cdot
\prod_{i=1}^b (q^{N'_\rho}- \alpha \cdot q^{\mu_i})}{2 \cdot q^{(a+b-1)(a+b+1)/4} \cdot |\text{SO}_{2m+1} (\F_q)|_{q'}}.
,}
which is a polynomial expression applied to $q^N$.

On the other hand, let us consider the values of $\text{dim} (\eta^V_{W,B} (\rho))$
for $V$ of dimension $2N$ as a function of $N$. We recall the description of endomorphism
algebra of $\omega [ V\otimes W]$ over $\text{Sp}(V)$ given in Section 2 of \cite{TotalHoweI}:
Considering the Schr\"{o}dinger model of the oscillator representation, there is an isomorphism
between the endomorphism algebra and the space of $\text{Sp}(V)$-fixed points in $\C (V\otimes W)$
\beg{EndIsOrbits}{(\text{End}_{\text{Sp}(V)} (\omega [ V\otimes W]), \circ) \cong (\C (V\otimes W)^{\text{Sp}(V)}, \star),}
where $\star$ is defined by
$$(v_1 \otimes w_1) \star (v_2 \otimes w_2) = \psi (\frac{S(v_1, v_2) \cdot B(w_1, w_2)}{2}) \cdot (v_1 \otimes w_1 + v_2 \otimes w_2)$$
(here $\psi$ denotes the non-trivial additive character corresponding to $1 \in \F_q^\times$,
under our identification of $\F_q$ with its Pontrjagin dual).
To consider the eta correspondence $\eta^V_{W,B}$, in \cite{TotalHoweI}
we consider $\omega [ V\otimes W]$ as a degree $\text{dim} (W)$ tensor product of oscillator
representations $\omega_{a_1} [V] \otimes \dots \otimes \omega_{a_n} [V]$ (considering
$B$ to be equivalent to the symmetric bilinear form corresponding to a diagonal
matrix with entries $a_1, \dots ,a_n$). This essential corresponds to writing out $V\otimes W$
as a direct sum of $n$ copies of $V$. Therefore we also view \rref{EndIsOrbits} as describing
\beg{EndTensorOms}{\text{End}_{\text{Sp}(V)} (\omega_{a_1} [V] \otimes \dots \otimes \omega_{a_n} [V])).}
We note that as long as $N \geq n$, the right hand side of \rref{EndIsOrbits}, as an algebra,
is stable and does not depend on $N$.
Therefore the same linear combination of $n$-tuples of $V$ vectors 
in the right hand side of \rref{EndIsOrbits} describes the idempotent
with image $\eta^V_{W,B} (\rho)$ for any choice of $N \geq n$. 
In particular, the dimension of $\eta^V_{W,B} (\rho)$ (expressible as
the trace of this idempotent in \rref{EndTensorOms}
for $V$ of dimension $2N$, is also polynomial in $q^N$, since, considering
one tensor factor at a time, trace of a linear combination
of $V$-vectors $(v)$ as an endomorphism of $\omega_{a_i} [V]$ is computed according to
$$tr ((v)) = \begin{cases}
0 \text{ if } v \neq 0\\
q^N \text{ if } v= 0
\end{cases}.$$
Hence, since this polynomial agrees with the polynomial \rref{PolynomialPhiDim}
for infinitely many values i.e., when applied to $q^N$ for $N$ large enough, they must in fact
always agree. 
Therefore, we obtain \rref{DimEtaIsDimPhi} for every $N \geq n$. 

Combining this with the results of the previous subsection which confirm
that the semisimple and sign parts of the classification data for 
$\eta^V_{W,B} (\rho)$ and $\phi^V_{W,B} (\rho)$ always match, we obtain that the unipotent
parts must match also (since every symbol has a different dimension). Therefore, we obtain that
$$\eta^V_{W,B} (\rho) = \phi^V_{W,B} (\rho),$$
by Lusztig's parametrization of irreducible representations, as claimed.

\end{proof}

\section{An Explicit Example: The case of $\text{SL}_2 (\F_q)$}\label{SL2Sect}

Consider, for example, the case of $N = 1$ (i.e. $\text{Sp}_2 (\F_q) = \text{SL}_2 ( \F_q)$), for $n = 2m+1$.
The oscillator representation $\omega [ \F_q^2]$ is $q$-dimensional, and decomposes along
the central $\Z/2$-action into pieces
$$\omega [ \F_q^2] = \omega^+ [\F_q^2] \oplus \omega^- [\F_q^2]$$
of dimension $(q+1)/2$, $(q-1)/2$, respectively.
Applying Lemma \ref{LemmaPartOddOrthoStabTopDim}
and Proposition \ref{PropFinalOddOrthoStabTopDim} above
gives that the top part of $ \omega [ \F_q^2 \otimes W]$ has dimension
$$q^{2m+1} - (q+1) = q \cdot (q^{2m}-1) -1.$$

Consider representations $\rho$ of $\text{SL}_2 (\F_q)$. The irreducible representations
are patametrized classification data consists
of the data of a conjugacy class of a semisimple element 
$$s \in \text{SO}_3 (\F_q) = \text{SL}_2^*(\F_q),$$
a unipotent representation $u$ of $(Z_{\text{SO}_3 (\F_q)} (s)^\circ)^*$,
and an additional choice of sign when $s$ has $-1$ eigenvalues.
There are $(q-3)/2$, resp. $(q-1)/2$, conjugacy classes $(s)$ (corresponding to having eigenvalues
$\{ \lambda, \lambda^{-1} \} \subseteq \mu_{q-1}\smallsetminus \{\pm1\}$, resp. $\mu_{q+1} \smallsetminus\{\pm 1\}$) with
$$(Z_{\text{SO}_3 (\F_q)} (s)^\circ)^* = U_1^+ (\F_q), \text{ resp. } U_1^- (\F_q),$$
whose only unipotent representation is trivial, and whose corresponding $\text{SL}_2 (\F_q)$-representation
then has dimension
$$\text{dim} (\rho^{\text{Sp}_2 (\F_q)}[s, 1]) = q+1 , \text{ resp. } q-1.$$
There is a single choice of semisimple conjugacy class
$(\sigma_1^\pm)$ each with 
the identity component of its centralizer isomorphic to 
$Z_{\text{SO}_3 (\F_q)} (s)^\circ = \text{SO}_2^\pm (\F_q)$
(corresponding to having eigenvalue $-1$ with multiplicity two, with sign determined by
the placement of the last eigenvalue $1$, depending on the presentation of the form defining
$\text{SO}_3 (\F_q)$), which again has only the trivial unipotent representation, giving representations
of dimension
$$\text{dim} (r^{\text{Sp}_2 (\F_q)}[\sigma_1^\pm, 1, +1]) = \text{dim} (
r^{\text{Sp}_2 (\F_q)}[\sigma_1^\pm, 1, -1]) = (q\pm 1)/2.$$
Finally, only $(s) = (I)$ has (the identity component of) its centralizer isomorphic the full $\text{SO}_3 (\F_q)$, which has two non-trivial unipotent
representations corresponding to symbols ${1 \choose \emptyset}$, ${0<1 \choose 1}$, of dimesions
$1$ and $q$, respectively.

We call the $s$ with no $-1$ eigenvalues the ``level $0$" choices. Call the other choices
of $s$ the ``level $1$" choices.
For the level $0$ choices of $s$, we have that
$$Z_{\text{Sp}_{2m} (\F_q)} (\psi (s))^\circ = (Z_{\text{SO}_3 (\F_q)} (s)^\circ)^* \times \text{Sp}_{2(m-1)} (\F_q),$$
with $\psi (u)$ defined as the representation corresponding to $u$ of the first factor, tensored with the
trivial representation of $\text{Sp}_{2(m-1)} (\F_q)$.
We assign the central sign describing the action of
of according to the discriminant of the form on $W$ and the quadratic
character
This fully describes $\zeta (r^{\text{Sp}_2 (\F_q)}[(s), 1])$ for the level $0$ $s$, and we find
$$\begin{array}{c}
\displaystyle \text{dim} (\zeta (r^{\text{Sp}_2 (\F_q)}[(s),1])) =\\
 \displaystyle \text{dim} (r^{\text{Sp}_2 (\F_q)}[(s),1]) \frac{|\text{SO}_{2m+1} (\F_q)|_{q'} }{|\text{Sp}_{2(m-1)}(\F_q)|_{q'} \cdot |\text{SO}_3 (\F_q)|_{q'}} =\\
\displaystyle \text{dim} (r^{\text{Sp}_2 (\F_q)}[(s),1]) 
\frac{q^{2m} -1}{q^2 -1}
\end{array}$$

For both level $1$ choices of $s$ (in this case, precisely $(s) = (\sigma_1^\pm)$), we have
$Z_{\text{Sp}_{2m} (\F_q)} (\psi (s)) = \text{Sp}_{2m} (\F_q),$
and we need to assign two choices of unipotent representations in both cases of the sign.
We alter the trivial representation of $\text{SO}_2^+ (\F_q)$ (corresponding to the symbol 
$1 \choose 0$ of rank 1, type $D$) by adjoining the coordinate $m$ to obtain the two choices
of symbols 
$$\textstyle { 1< m \choose 0}, \hspace{3mm} {1 \choose 0 <m},$$
describing unipotent representations of $\text{SO}_{2m+1} (\F_q)$.
Similarly, we alter the trivial representation of $\text{SO}_2^- (\F_q)$ (corresponding to the symbol 
${0<1 \choose \emptyset}$ of rank 1, type ${}^2 D$) by adjoining the cooedinate $m$
to obtain the two choices of symbols
$$\textstyle {0<1<m \choose \emptyset}, \hspace{3mm} {0<1 \choose m}.$$
Therefore, for level $1$ representations of $\text{SL}_2(\F_q)$, we have
$$\begin{array}{c}
\displaystyle \text{dim} (\zeta(r^{\text{Sp}_2 (\F_q)}[\sigma_1^+, 1, \pm1]) ) = \frac{(q^m \pm 1)(q^m\mp q)}{2 (q -1)}\\
\\
\displaystyle \text{dim} (\zeta(r^{\text{Sp}_2 (\F_q)}[\sigma_1^-, 1, \pm1]) ) = \frac{(q^m \pm 1)(q^m\pm q)}{2 (q +1)}.
\end{array}
$$

We may now apply our general combinatorial argument, but this case is small enough to verify directly.
Indeed, we can explicitly write out
$$
\begin{array}{c}
\displaystyle \sum_{\rho \in \widehat{\text{SL}_2(\F_q)}} \text{dim} (\rho) \cdot \text{dim} (\zeta (\rho))=\\
\\
$\resizebox{0.9\textwidth}{!}{$\displaystyle \frac{q-3}{2} \cdot \frac{(q+1)^2 (q^{2m}-1)}{q^2-1}+\frac{q-1}{2} \cdot \frac{(q-1)^2 (q^{2m}-1)}{q^2-1}+ \frac{(1+q^2)(q^{2m}-1)}{q^2-1}$}$\\
\\
\displaystyle + \frac{(q+1)(q^m+1)(q^m-q)}{4(q-1)}+ \frac{(q+1)(q^m-1)(q^m+q)}{4(q-1)}\\
\\
\displaystyle + \frac{(q-1)(q^m+1)(q^m+q)}{4(q+1)}+ \frac{(q-1)(q^m+1)(q^m+q)}{4(q+1)}
\end{array}
$$
(the first row corresponds to the level $0$ $\rho$, the second row corresponds to $\rho$ from $(s) = (\sigma_1^+)$, and the third row corresponds to $\rho$ from $(s) = (\sigma_1^-)$),
and verify that it equals $q(q^{2m}-1) -1$.

\end{document}